\documentclass[12pt, a4paper, twoside]{article}

\usepackage{exscale, amsmath, amssymb, amsthm, dsfont, amsfonts, latexsym, url}
\usepackage[curve,matrix,arrow,ps]{xy}
\usepackage[T1]{fontenc}
\usepackage[hypertex]{hyperref}

\newtheorem{definition}{Definition}[section]
\newtheorem{theorem}[definition]{Theorem}
\newtheorem{proposition}[definition]{Proposition}
\newtheorem{lemma}[definition]{Lemma}
\newtheorem{corollary}[definition]{Corollary}

\newtheorem{deflemma}[definition]{Definition-Lemma}

\newcommand{\nd}{\noindent}

\newcommand{\dC}{{\mathds C}}
\newcommand{\dQ}{{\mathds Q}}

\newcommand{\dN}{{\mathds N}}

\newcommand{\dR}{{\mathds R}}

\newcommand{\dZ}{{\mathds Z}}
\newcommand{\dP}{{\mathds P}}

\newcommand{\de}{{\mathds 1}}

\newcommand{\cA}{\mathcal{A}}

\newcommand{\cC}{\mathcal{C}}

\newcommand{\cE}{\mathcal{E}}

\newcommand{\cG}{\mathcal{G}}
\newcommand{\cH}{\mathcal{H}}

\newcommand{\cK}{\mathcal{K}}
\newcommand{\cL}{\mathcal{L}}
\newcommand{\cM}{\mathcal{M}}
\newcommand{\cN}{\mathcal{N}}
\newcommand{\cO}{\mathcal{O}}

\newcommand{\cQ}{\mathcal{Q}}
\newcommand{\cR}{\mathcal{R}}

\newcommand{\cT}{\mathcal{T}}
\newcommand{\cU}{\mathcal{U}}
\newcommand{\cV}{\mathcal{V}}

\newcommand{\cbl}{{\check D_{BL}}}
\newcommand{\cblp}{{\check D_{BL}^{pp}}}
\newcommand{\TERP}{(H,H'_\dR,\nabla,P,w)}
\newcommand{\VB}{\textit{VB}}

\DeclareMathOperator{\Spec}{\textup{Spec}\,}
\DeclareMathOperator{\Sp}{\textup{Sp}}

\DeclareMathOperator{\Aut}{\textup{Aut}}
\DeclareMathOperator{\diag}{\textup{diag}}

\DeclareMathOperator{\Gr}{Gr}

\DeclareMathOperator{\id}{id}

\DeclareMathOperator{\rank}{rank}

\DeclareMathOperator{\tr}{\textup{Tr}}

\DeclareMathOperator{\GI}{{\cG}^\mathit{I}}
\DeclareMathOperator{\GII}{\cG^\mathit{II}}
\DeclareMathOperator{\GIII}{\cG^\mathit{III}}
\DeclareMathOperator{\GIV}{\cG^\mathit{IV}}
\DeclareMathOperator{\GIw}{\widetilde{\cG}^\mathit{I}}
\DeclareMathOperator{\GIIw}{\widetilde{\cG}^\mathit{II}}
\DeclareMathOperator{\GIIIw}{\widetilde{\cG}^\mathit{III}}
\DeclareMathOperator{\GIVw}{\widetilde{\cG}^\mathit{IV}}
\DeclareMathOperator{\GIh}{\cG^{\mathit{I},hor}}
\DeclareMathOperator{\GIIh}{\cG^{\mathit{II},hor}}
\DeclareMathOperator{\GIIIh}{\cG^{\mathit{III},hor}}
\DeclareMathOperator{\GIVh}{\cG^{\mathit{IV},hor}}
\DeclareMathOperator{\Ks}{{\mathcal{K}^\mathit{sp}}}
\DeclareMathOperator{\CanD}{\cC^\mathit{an}_\mathit{\check{D}^{pp}_{BL}}}
\DeclareMathOperator{\CanM}{\cC^{an}_\mathit{M}}
\DeclareMathOperator{\DcPHS}{\check{\mathit{D}}_{\mathit{PHS}}}
\DeclareMathOperator{\DPHS}{\mathit{D}_{\mathit{PHS}}}
\DeclareMathOperator{\DcPMHS}{\check{\mathit{D}}_{\mathit{PMHS}}}
\DeclareMathOperator{\DPMHS}{\mathit{D}_{\mathit{PMHS}}}
\DeclareMathOperator{\KS}{\mathit{KS}}

\begin{document}

\title{Curvature of classifying spaces for Brieskorn lattices}

\author{Claus Hertling\and Christian Sevenheck}

\date{\today}

\maketitle

\begin{abstract}
We study $tt^*$-geometry on the classifying space
for regular singular TERP-structures, e.g., Fourier-Laplace transformations
of Brieskorn lattices of isolated hypersurface singularities.
We show that (a part of) this classifying space can be canonically
equipped with a hermitian structure. We derive an estimate for
the holomorphic sectional curvature of this hermitian metric, which is
the analogue of a similar result for classifying spaces of
pure polarized Hodge structures.
\end{abstract}

\renewcommand{\thefootnote}{}
\footnote{
C.H. acknowledges partial support by the ESF research grant MISGAM.
}

\section{Introduction}
\setcounter{equation}{0}
\label{secIntroduction}

In this paper, we study a generalization of variations of Hodge structures
and the associated period maps. These generalizations are called TERP-structures;
they first appeared under the name topological-antitopological fusion
(also called $tt^*$-geometry) in \cite{CV1,CV2,Du}
and were rigourously defined and studied in \cite{He4} and \cite{HS1}.

An important situation where TERP-structures naturally occur, is the
theory of ($\mu$-constant families) of isolated hypersurface singularities, and more specifically,
the Fourier-Laplace transformation of their Brieskorn lattices.
In this case the TERP-structures are regular singular. Irregular
TERP-structures arise by a similar though more general construction where
the initial object is a regular function on an affine variety. These
functions appear as mirror partners of the quantum cohomology algebra of
smooth projective varieties or more generally, orbifolds. It is a challenging
problem to study the induced TERP-structures on the quantum cohomology side,
although progress seems to have been made very recently in this direction (\cite{Ir}).
Let us notice that TERP-structures are intimately related to the theory
of harmonic bundles, via the so called twistor structures, i.e. (families of) holomorphic
bundles on $\dP^1$. Any TERP-structure gives rise to a twistor which is
called pure if it is a trivial bundle on $\dP^1$ and pure polarized if
a naturally defined hermitian metric on its space of global sections is
positive definite. This has to be seen as a generalization of the notion
of variations of (pure polarized) Hodge structures.
By a basic result of Simpson (\cite{Si5}), variations of pure polarized
twistor structures are equivalent to harmonic bundles on the parameter space.
Given a variation of TERP-structures on a complex manifold, one obtains
a variation of pure polarized twistor structures resp. a harmonic bundle
on an open subset of this manifold, which is a union of connected
components of the complement of a real analytic subvariety.
Notice also (\cite{Sa8}) that the TERP-structure of a tame function on an affine
manifold is always pure polarized.

The main topic of this paper are the classifying spaces that appear as
targets of period maps of variations of regular singular TERP-structures.
In fact, these spaces were already investigated under a different name (as classifying spaces of
Brieskorn lattices) in \cite{He2}. The main new point treated here
is to show how $tt^*$-geometry arises on the classifying spaces and to prove
the analogue of a crucial result in classical Hodge theory (see \cite{Sch}, \cite{GSch1},
\cite{GSch2} and \cite{De3}): the negativity of the sectional curvature of the
Hodge metric in horizontal directions. Similarly to the situation
in Hodge theory, we expect this result to be a cornerstone in
the study of the above mentioned period maps. We prove a few
quite direct consequences of our result at the end of this paper.

Let us give a short overview on the content of this article. In section \ref{secClassSpace}, we recall the
basic definitions both of variations of TERP-structures and of the classifying
spaces for Brieskorn lattices resp. regular singular TERP-structures. In order to do that,
we also recall the construction of the polarized mixed Hodge structure and its cohomological
invariants, the spectral numbers associated to a regular singular TERP-structure.
In section \ref{secTangent}, we construct a Kodaira-Spencer map from the tangent bundle of the
classifying space to some auxiliary bundle which gives a local trivialization
of the tangent bundle needed later. In particular, this induces a positive definite
hermitian metric on the pure polarized part of the classifying space.
We also consider the subsheaf of the tangent bundle of the classifying
space consisting of horizontal directions. Contrary to the case of
Hodge structures, it is not locally free in general. Finally, in section \ref{secHolSectCurv}
the main result of the paper is stated and proved. The proof is considerably more complicated
than in the case of Hodge structures as the classifying spaces of
TERP-structures/Brieskorn lattices are  not homogenous. We finish the paper by
deducing from our main theorem a rigidity result for variations of TERP-structures
on affine spaces.
\paragraph{Notations:} For a complex manifold $X$, we write $\cE\in\VB_X$ for a locally free
sheaf of $\cO_X$-modules $\cE$, the associated vector bundle is
denoted by $E$. If $E$ comes equipped with a flat
connection, we denote by $E^\nabla$ the corresponding local system.

\section{Classifying spaces}
\setcounter{equation}{0}
\label{secClassSpace}

In this section we introduce the classifying spaces of regular singular TERP-structures
which were considered, under a different name, in \cite{He2}. We start
by recalling very briefly the basic definition of a TERP-structure and
some of its associated data. After this, we give the definition
of the classifying spaces.

For the following basic definition we also refer to
\cite{He4} and \cite{HS1}.
\begin{definition}
Let $X$ be a complex manifold and $w$ an integer. A variation of TERP-structures
on $X$ of weight $w$ consists of a holomorphic vector bundle $H$ on $\dC\times X$, an integrable connection
$\nabla:\cH\rightarrow \cH\otimes\Omega^1_{\dC\times X}(*\{0\}\times X)$,
a flat real subbundle $H'_\dR$ of maximal rank of the restriction $H':=H_{|\dC^*\times X}$
and a flat non-degenerate $(-1)^w$-symmetric pairing $P:\cH'\otimes
j^*\cH' \rightarrow \cO_{\dC^*\times X}$, where
$j(z,t):=(-z,t)$, subject to the following conditions:
\begin{enumerate}
\item
$\nabla$ has a pole of type one (also called of Poincar\'e rank one) along $\{0\}\times X$,
i.e., the sheaf $\cH$ is stable under $z^2\nabla_z$ and $z\nabla_T$ for any
$T\in p^{-1}\cT_X$, where $p:\dC\times X\twoheadrightarrow X$.
\item
$P$ takes values in $i^w\dR$ on $H'_\dR$
\item
$P$ extends as a non-degenerate pairing $P:\cH\otimes j^*\cH \rightarrow z^w\cO_{\dC\times X}$, in
particular, it induces a non-degenerate symmetric pairing $[z^{-w}P]:\cH/z\cH\otimes\cH/z\cH
\rightarrow\cO_X$.
\end{enumerate}
$\TERP$ is called regular singular, if $(\cH,\nabla)$ is regular singular along
$\{0\}\times X$, i.e, if sections of $\cH$ have moderate growth along $\{0\}\times X$
compared to flat sections of $\cH'$.
\end{definition}
The case $X=\{pt\}$ is referred to as a single TERP-structure. There is a canonically
associated set of data, which we call ``topological''.
\begin{definition}\label{defTopDataTERP}
Let $\TERP$ be a TERP-structure, then we put
$$
H^\infty:=\{\textup{flat multivalued sections of }\cH'\}.
$$
We let $H^\infty_\dR$ be the subspace of real flat multivalued sections,
then $H^\infty_\dR$ comes equipped with the monodromy endomorphism $M\in\mathit{Aut}(H^\infty_\dR)$,
which decomposes as $M=M_s\cdot M_u$ into semi-simple and unipotent part. Let
$H^\infty:=\oplus H^\infty_\lambda$ be the decomposition into generalized eigenspaces with respect to $M$.

We restrict here to the case where all eigenvalues have absolute
value $1$, as this is automatically the case for \emph{mixed} TERP-structures,
as defined in definition \ref{defMixedTERP}. We
put $H^\infty_{\neq 1}:=\oplus_{\lambda \neq 1} H^\infty_\lambda$, so that
$H^\infty=H^\infty_1 \oplus H^\infty_{\neq 1}$, and let $N:=\log(M_u)$
be the nilpotent part of $M$.

$P$ induces a polarizing form $S$ on $H^\infty$ defined as follows: First note
that $P$ corresponds (after a counter-clockwise shift in the second
argument) to a pairing $L$ on the
local system $(H')^\nabla$, then given $A,B\in H^\infty$, we put
$S(A,B):= (-1)(2\pi i)^w L(A, t(B))$ where
$$
t(B)=
\left\{
\begin{array}{c}
(M-\mathit{Id})^{-1}(B) \;\;\;\;\;\; \forall B\in H^\infty_{\neq 1}\\ \\
-(\sum\limits_{k\geq1}\frac{1}{k!}N^{k-1})^{-1}(B)\;\;\;\;\;\; \forall B\in H^\infty_1.
\end{array}
\right.
$$
$S$ is nondegenerate, monodromy invariant, $(-1)^w$-symmetric on $H^\infty_1$, $(-1)^{w-1}$-symmetric
on $H^\infty_{\neq 1}$, and it takes real values on $H^\infty_\dR$ \cite[Lemma 7.6]{He4}. We call
the tuple $(H^\infty, H^\infty_\dR, M, S, w)$ the topological data of $\TERP$.
\end{definition}
Note that by \cite[lemma 5.1]{HS1} the topological data are equivalent to the
data $(H',H'_\dR, \nabla, P, w)$.

Let us now suppose that $\TERP$ is regular singular. Then the following classical objects
will play a key role in the sequel of this paper.
\begin{definition}\label{defHodgeFiltSpectrum}
Let $\TERP$ be a regular singular TERP-structure.
\begin{enumerate}
\item
Define for any $\alpha\in\dC$, $C^\alpha:=z^{\alpha\mathit{Id}-\frac{N}{2\pi i}}H^\infty_{e^{-2\pi i \alpha}}
\subset i_*(\cH')_0$ to be the space of elementary sections of $H'$ of order $\alpha$.
Let $V^{\alpha}$ (resp. $V^{>\alpha}$) the free $\cO_\dC$-module generated by elementary sections
of order at least (resp. strictly greater than) $\alpha$, i.e. $V^\alpha:=\sum_{\beta\geq\alpha} \cO_\dC C^\beta$
and $V^{>\alpha}:=\sum_{\beta>\alpha} \cO_\dC C^\beta$. The meromorphic bundle
(i.e., locally free $\cO_\dC[z^{-1}]$-module) $V^{>-\infty}$ is defined as $V^{>-\infty}:=\bigcup_\alpha V^\alpha$.
Any $V^\alpha$ and $V^{>\alpha}$ is a lattice inside $V^{>-\infty}$ and the decreasing
filtration $V^\bullet$ is called Kashiwara-Malgrange-filtration (or V-filtration) of $V^{>-\infty}$.

Notice that the objects $C^\alpha, V^\alpha, V^{>\alpha}$ and $V^{>-\infty}$ only depend
on the topological data of the TERP-structure, i.e. on $(H',H'_\dR, \nabla, P, w)$, but not
on the extension $H$ of the vector bundle $H'$ on $\dC^*$ to a vector bundle on $\dC$.
\item
The regularity assumption on $(H,\nabla)$ can be rephrased by saying that $\cH\subset V^{>-\infty}$.
The V--filtration induces a filtration on $\cH$, which is used to define a decreasing
filtration on the space $H^\infty$ in the following way.
Define for any $\alpha\in(0,1]+i\dR$
$$
F^pH^\infty_{e^{-2\pi i \alpha}}:=z^{p+1-w-\alpha+\frac{N}{2\pi i}}Gr_V^{\alpha+w-1-p}\cH,
$$
then $F^\bullet$ is a decreasing exhaustive filtration on $H^\infty$.
We will use a twisted version of this filtration, which is obtained as
$\widetilde{F}^\bullet:=G^{-1}F^\bullet$, where
$G:=\sum_{\alpha \in(0,1]+i\dR} G^{(\alpha)} \in
\Aut\left(H^\infty=\oplus_{\alpha}
H^\infty_{e^{-2\pi i\alpha}}\right) $ is defined as follows (see \cite[(7.47)]{He4}):
$$
G^{(\alpha)} := \sum_{k\geq 0}\frac{1}{k!}\Gamma^{(k)}(\alpha)
\left( \frac{-N}{2\pi i}\right)^k
=: \Gamma \left(\alpha\cdot id - \frac{N}{2\pi i}\right) .
$$
Here $\Gamma^{(k)}$ is the $k$-th derivative of the gamma function.
In particular, $G$
depends only on $H'$ and induces the identity on $\Gr^W_\bullet$
where $W_\bullet$ is the weight filtration of the nilpotent endomorphism $N$.
Note that the restriction of $W_\bullet(N)$ to $H^\infty_1$ is by definition centered around
$w$, and the restriction to $H^\infty_{\neq 1}$ is centered around $w-1$.

$\widetilde{F}^\bullet$ is the Hodge filtration of Steenbrink
if the TERP-structure is defined by an isolated hypersurface
singularity.

As a matter of notation, we also write
$\widetilde{F}^\bullet_\cH$ for the filtration $\widetilde{F}^\bullet$
on $H^\infty$ defined by $\cH$.
\item
The $V$-filtration is also used to define the spectrum of a regular
singular TERP-structure $\TERP$. Namely, let
$\Sp(H,\nabla) = \sum_{\alpha \in\dC}d(\alpha)\cdot\alpha\in\dZ[\dC]$
where
\begin{equation}\label{eqSpectrum}
d(\alpha):=\dim_\dC\left(\frac{Gr^\alpha_V \cH}{Gr^\alpha_V z\cH}\right)
=\dim_\dC\Gr_F^{\lfloor w-\alpha\rfloor} H^\infty_{e^{-2\pi i \alpha}}.
\end{equation}
It is a tuple of $\mu$ complex numbers $\alpha_1 \leq \ldots \leq \alpha_\mu$.
By definition, $d(\alpha)\neq 0$ only if $e^{-2\pi i \alpha}$ is an eigenvalue of $M$.
We have the symmetry property $\alpha_1 + \alpha_\mu = w$.
In most applications the eigenvalues of $M$ are roots of unity so that the spectrum
actually lies in $\dZ[\dQ]$.
\end{enumerate}
\end{definition}

The following notion is quoted from \cite{HS1},
where it is shown to correspond to ``nilpotent orbits'' of TERP-structures.
\begin{definition}\label{defMixedTERP}
A regular singular TERP-structure $\TERP$ of weight $w$ is called mixed if
the tuple
$$
(H^\infty_{\neq 1}, (H^\infty_{\neq 1})_\dR, -N, S, \widetilde{F}^\bullet)
\mbox{ resp. }
(H^\infty_1, (H^\infty_1)_\dR, -N, S, \widetilde{F}^\bullet)
$$
is a polarized mixed Hodge structure of weight $w-1$ resp. of weight $w$.
We refer to \cite{He4} or \cite{HS1} for the notion of
a polarized mixed Hodge structure (PMHS for short) used here. The data here are
$M_s$-invariant. In \cite[lemma 5.9]{HS1} it is shown that any semi-simple
automorphism of a PMHS has eigenvalues in $S^1$. This is compatible with our
assumptions in definition 2.2. that $M_s$ has all its eigenvalues in $S^1$ and
justifies this assumption.
\end{definition}

Next we reformulate the definitions of the classifying spaces
$D_\mathit{BL}$ resp.  $\DPMHS$ for Brieskorn lattices resp. PMHS from
\cite{He2} in terms of regular singular TERP-structures.
We start with a PMHS of one single weight $w$ with a semisimple automorphism.
As we have seen, a mixed TERP-structure defines
a sum of PMHS's of different weights on $H^\infty_1\oplus H^\infty_{\neq 1}$,
so that later we need a slight adjustment of this situation (this is done
in definition \ref{defClassSpaces}).

The next lemma gives an equivalence of conditions for a filtration
to induce a PMHS.
\begin{lemma}\label{lemCondInDcPMHS}
Let $(H^\infty,H^\infty_\dR,N,S,F^\bullet_0)$ be a PMHS of weight $w$
and let $M_s$ be a semisimple automorphism of it.
Then the eigenvalues of $M_s$ are elements of $S^1$.
Let $W_\bullet$ be the weight filtration centered at weight $w$ which is
induced by $N$. Let $P_{l}$ be the primitive subspace
$P_l:=\ker (N^{l-w+1}:\Gr^W_{l}\to \Gr^W_{2w-l-2})$ of $\Gr^W_l$
(for $l\geq w$)
and let $G_\dC$ be the group $G_\dC:=\Aut(H^\infty,N,S,M_s)$.
The primitive subspace $P_l$ decomposes into the eigenspaces of $M_s$,
$P_l=\bigoplus_\lambda P_{l,\lambda}$.
Then for any
$M_s$-invariant filtration $F^\bullet$ on $H^\infty$,
the following conditions are equivalent.
\begin{enumerate}
\item
$\dim F^pP_{l,\lambda}=\dim F^p_0P_{l,\lambda},
\ N(F^p)\subset F^{p-1}$, $F^pN^jP_l=N^jF^{p+j}P_l$,\\
$F^p\Gr^W_l=\bigoplus_{j\geq 0} F^pN^jP_{l+2j}$, $S(F^p,F^{w+1-p})=0$.
\item
There exists an $M_s$-invariant
common splitting $\widetilde{I}^{p,q}$ of $F^\bullet$ and
$W_\bullet$ with the properties in lemma 2.3 (a)--(d) in \cite{He2}.
\item
$F^\bullet$ is the image of $F^\bullet_0$ by an element of $G_\dC$.
\item
$\dim F^pP_{l,\lambda}=\dim F^p_0P_{l,\lambda},\ S(F^p,F^{w+1-p})=0$,
and all powers of $N$ are strict with respect
to $F^\bullet$.
\end{enumerate}
\end{lemma}

As to the proof, let us just remark that obviously 2. implies 1., 3., and 4.
The equivalence of 1., 2. and 3. is proved
in \cite[Ch. 2]{He2}. The only remaining point is that 4. implies 1.-3., which is
rather technical. As we will not use the characterization 4., it is skipped here.

Also the following is proved in \cite[chapter 2]{He2}.

\begin{proposition}\label{propClassSpacesOneWeight}
In the situation of lemma \ref{lemCondInDcPMHS}, consider the set
\begin{eqnarray*}
\DcPMHS(H^\infty,H^\infty_\dR,N,S,F^\bullet_0,M_s,w) :=
\big\{\textup{filtrations }F^\bullet H^\infty\ |\ F^\bullet H^\infty
\textup{ is }M_s\textup{-invariant}\\
\textup{and satisfies the equivalent conditions in lemma \ref{lemCondInDcPMHS}}\big\}.
\end{eqnarray*}
It is a complex homogeneous space on which $G_\dC$ acts transitively.
The set
\begin{eqnarray*}
\DPMHS(H^\infty,H^\infty_\dR,N,S,F^\bullet_0,M_s,w) :=
\big\{\textup{filtrations }F^\bullet H^\infty\ |\ F^\bullet H^\infty
\textup{ is }M_s\textup{-invariant,}\hspace*{2cm}\\
\dim F^pP_{l,\lambda}=\dim F^p_0P_{l,\lambda},
\ (H^\infty,H^\infty_\dR,N,S,F^\bullet,M_s)\textup{ is a PMHS of weight }w\big\}
\end{eqnarray*}
is an open submanifold of $\DcPMHS$ and a real homogeneous space with transitive
action by a certain real group lying in between $G_\dC$
and $G_\dR=\Aut(H^\infty_\dR,N,S,M_s)$.
It is a classifying space for the $M_s$-invariant PMHS with the same
discrete data as the reference PMHS defined by $F^\bullet_0$.

Consider for $l\geq w$ the pairing $S_{l-w}:=S(-,N^l-)$ on $\Gr^W_l$,
the primitive subspaces $P_l:=\mathit{Ker}(N^{l-w+1})\subset \Gr^W_l$.
They decompose as $P_l=\oplus_\lambda P_{l,\lambda}$ into
eigenspaces of $M_s$. Then we define for any $l\geq w$:
$$
\begin{array}{rcl}
\check{D}_l(H^\infty,H^\infty_\dR,N,S,F^\bullet_0,M_s,w)
 & := & \big\{\textup{filtrations }F^\bullet P_l \,|\,F^p P_l \textup{ is }M_s
\textup{-invariant}, \\
& & \dim F^p P_{l,\lambda}=
\dim F_0^p P_{l,\lambda},\, S_{l-w}(F^pP_l ,F^{l-p+1}P_l)=0\big\}, \\ \\
D_l(H^\infty,H^\infty_\dR,N,S,F^\bullet_0,M_s,w)
 & := & \left\{F^\bullet P_l \in \check{D}_l \,|\, F^\bullet P_l \textup{ gives
a PHS of weight }l\textup{ on }P_l \right\} ,\\ \\
\DcPHS(H^\infty,H^\infty_\dR,N,S,F^\bullet_0,M_s,w) & := & \prod\limits_{l\geq w} \check{D}_l(H^\infty,H^\infty_\dR,N,S,F^\bullet_0,M_s,w) , \\ \\
\DPHS(H^\infty,H^\infty_\dR,N,S,F^\bullet_0,M_s,w) & := & \prod\limits_{l\geq w} D_l(H^\infty,H^\infty_\dR,N,S,F^\bullet_0,M_s,w).
\end{array}
$$
$\DcPHS(H^\infty,H^\infty_\dR,N,S,F^\bullet_0,M_s,w)$
is a projective manifold and a complex homogenous space, $\DPHS(H^\infty,H^\infty_\dR,N,S,F^\bullet_0,M_s,w)$
is an open submanifold and a real homogenous space. The projection
$$
\begin{array}{rcl}
\pi_\mathit{PMHS}:\DcPMHS(H^\infty,H^\infty_\dR,N,S,F^\bullet_0,M_s,w) & \longrightarrow & \DcPHS(H^\infty,H^\infty_\dR,N,S,F^\bullet_0,M_s,w) \\ \\
F^\bullet & \longmapsto & \prod_{l\geq w} F^\bullet P_l
\end{array}
$$
is an affine fiber bundle with fibre isomorphic to $\dC^{N_\mathit{PMHS}}$ for
some $N_\mathit{PMHS}\in\dN\cup\{0\}$. The classifying space $\DPMHS$ is the
restriction of the total space of this bundle to $\DPHS$, in other words,
we have the following diagram of projections and inclusions
\begin{eqnarray}\label{diagClassSpaces1}
\begin{matrix}
\DcPMHS(H^\infty,H^\infty_\dR,N,S,F^\bullet_0,M_s,w) & \stackrel{\pi_\mathit{PMHS}}{\longrightarrow} & \DcPHS(H^\infty,H^\infty_\dR,N,S,F^\bullet_0,M_s,w)\\
\cup & & \cup \\
\DPMHS(H^\infty,H^\infty_\dR,N,S,F^\bullet_0,M_s,w) & \longrightarrow &  \DPHS(H^\infty,H^\infty_\dR,N,S,F^\bullet_0,M_s,w).
\end{matrix}
\end{eqnarray}
\end{proposition}
The next definition introduces the main objects of this paper, namely, the classifying
spaces of regular singular TERP-structures. The most direct way to fix the data needed
to define these spaces is to consider a reference TERP-structure $(H^{(0)},
H'_\dR, \nabla, P, w)$, which is supposed to be mixed.
This defines the set of discrete data needed, among them are the topological data of $(H^{(0)},
H'_\dR, \nabla, P, w)$ as well as its spectral numbers.
\begin{definition}\label{defClassSpaces}
Let $(H^{(0)},H'_\dR,\nabla,P,w)$ be a regular singular mixed TERP-structure.
Consider
its topological data $(H^\infty, H_\dR^\infty, M, S, w)$ as defined in definition
\ref{defTopDataTERP}. We also have the filtration $\widetilde{F}^\bullet_0:=\widetilde{F}^\bullet_{\cH^{(0)}}$ and
the spectral numbers of $\cH^{(0)}$ as in definition \ref{defHodgeFiltSpectrum}.
Then, using proposition \ref{propClassSpacesOneWeight}
we define
\begin{equation}
\begin{array}{l}
\DcPMHS :=  \DcPMHS(H^\infty_1,(H^\infty_\dR)_1,-N,S,\widetilde{F}^\bullet_0,\id,w)
\\ \\ \quad \quad \quad \quad \times \DcPMHS(H^\infty_{\neq 1},(H^\infty_\dR)_{\neq 1},-N,S,\widetilde{F}^\bullet_0,M_s,w-1)
\end{array}
\end{equation}
and similarly $\DPMHS$, $\DcPHS$, $\DPHS$.

Write, as before, $H':=H^{(0)}_{|\dC^*}$. Then we are interested in
all possible extensions of $H'$ to a vector bundle $H$ on $\dC$ making
$\TERP$ into a TERP-structure and such that the associated filtration
is an element in the space $\DcPMHS$ just defined. This leads to the following definition.
\begin{eqnarray}
\cbl & := & \big\{\textup{free }\cO_{\dC,0}\textup{-submodules }\cH_0
\textup{ of }V^{>-\infty}_0\,|\, (z^2\nabla_z)\cH_0\subset\cH_0, \\
\nonumber & & \;P(\cH_0,\cH_0)=z^w\cO_\dC,\widetilde{F}^\bullet_\cH\in\DcPMHS\big\} \\ \nonumber \\
D_\mathit{BL} & := & \left\{\cH_0\in\cbl\,|\,\widetilde{F}^\bullet_\cH\in\DPMHS\right\}
\end{eqnarray}
\end{definition}
Notice first that as any element $\cH_0$ in $\cbl$ defines an extension of
the fixed (flat) bundle $H'$ on $\dC^*$, one may rephrase the definition
of $\cbl$ by saying that its elements are  the bundles $H$ on $\dC$
extending $H'$ such that
$\TERP$ is a regular singular TERP-structure and such that the associated filtration
$\widetilde{F}^\bullet_\cH$ lies in the same classifying space as $F^\bullet_{\cH^{(0)}}$.
Similarly, $D_\mathit{BL}$ consists of bundles $H$ such that $\TERP$ is a regular
singular mixed TERP-structure with $\widetilde{F}^\bullet_{\cH^{(0)}}\in\DPMHS$.

We already remarked that $\DcPMHS, \DPMHS, \DcPHS, \DPHS$ are complex manifolds
and complex resp. real homogenous spaces. A priori, the above definition
describes $\cbl$ resp. $D_\mathit{BL}$ only as a set with no obvious topological or
analytical structure. However, one of the main results of \cite{He2} (see also
\cite[proof of theorem 12.8]{He3}) is that $\cbl$ has a natural structure of a complex
manifold and that the projection $\pi_\mathit{BL}:\cbl\rightarrow \DcPMHS$,
$\cH\mapsto\widetilde{F}^\bullet_{\cH}$ is a affine
fibre bundle over $\DcPMHS$ with fibres isomorphic to $\dC^{N_\mathit{BL}}$ for
some $N_\mathit{BL}\in\dN\cup\{0\}$.
Moreover $D_\mathit{BL}$ is
the restriction of the map $\pi_\mathit{BL}$ to $\DPMHS$. Notice that
contrary to $\DcPMHS$, $\DPMHS$, $\DcPHS$ and $\DPHS$,
the spaces $\cbl$ and $D_\mathit{BL}$ are not homogenous. However,
there is a good $\dC^*$-action on the fibers of
$\pi_\mathit{BL}$, the corresponding zero section
$\DcPMHS\hookrightarrow \cbl$
consists of the regular singular TERP-structures in $\cbl$
which are generated by elementary sections, see \cite[theorem 5.6]{He2}.

Notice also that it follows from the definition of the space $\cbl$ that
all elements $\cH \in \cbl$ have the same spectral numbers, namely those
of $\cH^{(0)}$. The reason for this is that the spectrum is determined by the
topological data (more precisely, by the eigenvalues of $M$) and the filtration
$\widetilde{F}^\bullet_\cH$ (namely, by its Hodge numbers, i.e.,
the dimensions $\dim\Gr^{\lfloor w-\alpha \rfloor}_F H^\infty_{e^{-2\pi i \alpha}}$,
see formula \eqref{eqSpectrum}). Moreover, the definition of
the space $\DcPMHS$ (condition 1. in lemma \ref{lemCondInDcPMHS})
shows that these numbers are constant for all $F^\bullet\in\DcPMHS$, namely they are equal to those of
$\widetilde{F}^\bullet_{\cH^{(0)}}$. We will denote all along this article
these spectral numbers by $\alpha_1, \ldots, \alpha_\mu$, where $\mu:=\dim_\dC(H^\infty)$.

The following diagram completes diagram \eqref{diagClassSpaces1} and visualizes
how all of the above defined manifolds are interrelated.
\begin{eqnarray}\label{diagClassSpaces2}
\begin{matrix}\cbl & \stackrel{\pi_\mathit{BL}}{\longrightarrow} & \DcPMHS &
\stackrel{\pi_\mathit{PMHS}}{\longrightarrow} & \DcPHS\\
\cup & & \cup & & \cup \\
D_\mathit{BL} & \longrightarrow & \DPMHS &
\longrightarrow &  \DPHS.
\end{matrix}
\end{eqnarray}

\section{The tangent bundle and horizontal directions}
\setcounter{equation}{0}
\label{secTangent}

This section gives a concrete description of the tangent bundle
of the manifold $\cbl$ using a Kodaira-Spencer map. This description
is used in an essential way in the curvature calculation in the next section.

Consider the flat bundle $H'\in \VB_{\dC^*}$ (with real structure) and
the pairing $P:\cH'\otimes j^*\cH'\rightarrow\cO_{\dC^*}$ which correspond to
$(H^\infty, H_\dR^\infty, S, M, w)$ by \cite[lemma 5.1]{HS1}. By abuse of notation,
we will denote by $H'$ (resp. $\cH'$ for its sheaf of sections) the pullback
of this bundle by the projection $\pi':\dC^*\times \cbl \rightarrow \dC^*$. This bundle
comes equipped with an integrable connection, which is the pullback of the original
connection of $H'\in\VB_{\dC^*}$.
Similarly, $P$ pulls back to a pairing $(\pi')^*P:\cH'\otimes j^*\cH' \rightarrow \cO_{\dC^*\times \cbl}$,
which we denote, by abuse of notation, again by $P$.
We also consider the pullbacks of the Deligne lattices $V^\alpha$ and $V^{>\alpha}$
and of the meromorphic bundle $V^{>-\infty}$ under the projection
$\pi:\dC\times\cbl\rightarrow\cbl$. We write
$\cV^\alpha$, $\cV^{>\alpha}$ and $\cV^{>-\infty}$ for the pull-backs
$\pi^*V^\alpha$, $\pi^*\cV^{>\alpha}$ and $\pi^*\cV^{>-\infty}$, respectively.
All of them are extensions of $\cH'$ to $\dC\times\cbl$, i.e.,
we have $\cV^\alpha,\cV^{>\alpha},\cV^{>-\infty}\subset i_*\cH'$, where
$i:\dC^*\times\cbl\hookrightarrow \dC\times\cbl$ is the inclusion.
Notice however that $\cV^\alpha$ and $\cV^{>\alpha}$ are $\cO_{\dC\times\cbl}$-locally
free, whereas $\cV^{>-\infty}$ is only $\cO_{\dC\times\cbl}[z^{-1}]$-locally free.
By definition, we have a connection operator
$$
\nabla:\cV^\alpha\longrightarrow\cV^\alpha\otimes\Omega^1_{\dC\times\cbl}(\log\,\{0\}\times\cbl),
$$
and similarly for $\cV^{>\alpha}$.

From the very definition of the space $\cbl$, we see that there is another naturally defined
extension of $\cH'$ to $\dC\times\cbl$, which we call $\cL$. It is the universal family of Brieskorn
lattices, i.e., $\cL_{|\dC\times\{t\}} = \cH(t)$ for any $t\in\cbl$.
It can equally be described by gluing the locally defined bundles over $\dC\times U$, $U\subset \cbl$, given by the bases
constructed in \cite[lemma 5.2 to theorem 5.6]{He2}. The pairing
$P$ has the property that $P(\cL,\cL)\subset z^w\cO_{\dC\times\cbl}$ and that $z^{-w}P$ defines a non-degenerate
symmetric pairing on $\cL/z\cL$. In particular, the original $P$ on $\cH'$ takes values in
$\cO_{\dC\times\cbl}[z^{-1}]$ when restricted to $\cV^{>-\infty}$.

By definition, $\cL$ comes equipped with a connection
$$
\nabla:\cL\rightarrow\cL\otimes
\left(z^{-2}\Omega^1_{\dC\times\cbl/\cbl}
\oplus
\Omega^1_{\dC\times\cbl/\dC}(*\{0\}\times\cbl)\right).
$$
Using the Deligne extensions $\cV^\alpha$, we can give a precise statement
on the pole order of $\nabla$ on $\cL$.
Define $n:=[\alpha_\mu-\alpha_1]$, note that if $n=0$ then the classifying space $\cbl$ consists
of a single element only, namely, the lattice $V^{\alpha_1}$.
\begin{lemma}
\label{lemPoleOrder-n}
Suppose that $n>0$. Then $\cL$ is stable under $z^n\nabla_X$ for any $X\in p^{-1}\cT_\cbl$,
i.e., we have a connection operator
$$
\nabla:\cL\longrightarrow\cL\otimes
\left(
z^{-2}\Omega^1_{\dC\times\cbl/\cbl}
\oplus
z^{-n}\Omega^1_{\dC\times\cbl/\dC}
\right)
$$
\end{lemma}
\begin{proof}
As explained before, any $\cH\in \cbl$ has
the spectral numbers $\alpha_1,\ldots,\alpha_\mu$. It follows
in particular hat $V^{\alpha_1} \supset \cH \supset V^{>\alpha_\mu-1}$.
The first inclusion is obvious, for the second, if we had any $s\in V^{>\alpha_\mu-1}$ which
is not a section of $\cH$, then there would be a $k\in\dN_{>0}$
with $z^k s \in \cH$ which implies (if we take a minimal such $k$) that the principal part of $z^k s$ does
not vanish in $Gr^V_\bullet(\cH/z\cH)$. In other words, we would get a spectral number larger then $\alpha_\mu$, which
is impossible. As these inclusions of lattices hold true at any point of $\cbl$, we have
$$
\cV^{\alpha_1}\supset\cL\supset\cV^{>\alpha_\mu-1}.
$$
Now let $s$ be any local section of $\cL$, then
$s\in \cV^{\alpha_1}$, so that
$$
\nabla_X(s)\subset \cV^{\alpha_1}
\subset \cV^{>-n+\alpha_\mu-1}=z^{-n}\cV^{>\alpha_\mu-1}\subset z^{-n}\cL.
$$
\end{proof}

The following lemma will be useful for the proof of lemma \ref{lemKS},
but it will also be used in the computations in the next
section.

\begin{lemma}\label{lemCanExtBase}
Consider a module $\cH\in \cbl$  and a local $\cO_\dC$-basis
$\underline{v}^{(0)}:=(v^{(0)}_1,\ldots,v^{(0)}_\mu)$ of $\cH$.
\begin{enumerate}
\item
There exists a small neighborhood $U_1\times U_2$ of
$(0,\cH)\in \dC\times\cbl$ and a unique basis $\underline{v}:=(v_1,\dots,v_\mu)$
of $\cL_{|U_1\times U_2}$ which is an extension of
$\underline{v}^{(0)}$ and which satisfies
\begin{eqnarray}\label{v=v_0}
\underline{v}=\underline{v}^{(0)}\cdot
\left(\de_\mu+\sum_{k=1}^n z^{-k} C_k\right)
\end{eqnarray}
where  $C_k\in M(\mu\times \mu,\cO_\cbl(U_2))$. Here $\underline{v}^{(0)}$ is extended to a section
in $\pi^{-1}V^{>-\infty}\subset\cV^{>-\infty}$, and this equation holds
in $\cV^{>-\infty}$, i.e., meromorphically along $z=0$.
\item
If $z^{-w}P((\underline{v}^{(0)})^{tr},\underline{v}^{(0)})$ is a constant $\mu\times\mu$-matrix
then so is $z^{-w}P((v)^{tr},v)$.
\item
If $\underline{v}^{(0)}$ is a good basis for $\cH$ in the
sense of \cite{SM} (i.e.,
if it projects to a basis of the vector space $\Gr^V_\bullet(\cH/z\cH)$)
then $\underline{v}$ is a good basis for $\cL_{|U_1\times U_2}$. Here
being a good basis for $\cL_{|U_1\times U_2}$ can be
expressed in two equivalent ways: Either we require that
for any $t\in U_2$, the restriction of $\underline{v}$ to
$U_1\times \{t\}$ is a good basis of the restriction $\cL_{|U_1\times\{t\}}$
or we ask that $\underline{v}$ projects to a $\cO_{U_2}$-basis of
$\Gr^\cV_\bullet(\cL/z\cL)$. These requirements are equivalent as the latter module
is locally free due to the fact that the spectral numbers of $\cL_{|\dC\times\{t\}}$ are the same for each $t\in U_2$.
\end{enumerate}
\end{lemma}

\begin{proof}
\begin{enumerate}
\item
Consider a holomorphic extension $\underline{v}'=(v'_1,...,v'_\mu)$
of $\underline{v}^{(0)}$ in a suitable neighborhood
$\Delta_\varepsilon\times U'_2$ of $(0,\cH)$. The matrix
$\Psi$ with $\underline{v}'=\underline{v}^{(0)}\cdot \Psi$ is holomorphic and invertible on
$\Delta^*_\varepsilon \times U'_2$ and
defines a cocycle in
$H^1(\dP^1\times U'_2,
\mathit{GL}(\mu,\cO^*_{\dP^1\times U'_2}))$
and thus a vector bundle on $\dP^1\times U'_2$. Because the restriction
to $\dP^1\times \{\cH\}$ is trivial, the restriction to $\dP^1\times U_2$
for some $U_2\subset U_2'$ is a family of trivial vector bundles
on $\dP^1$. The Birkhoff factorization (see, e.g., \cite[proposition 4.1]{Mal2}) yields
unique matrices $\Psi_0\in
\Gamma(\Delta_\varepsilon\times U_2,\cO^*_{\Delta_\varepsilon\times U_2})$
and $\Psi_\infty\in \Gamma((\dP^1\backslash\{0\})\times U_2,
\cO^*_{(\dP^1-\{0\})\times U_2})$ with
$\Psi_\infty|_{\{\infty\}\times U_2}=\de_\mu$ and
$\Psi=\Psi_\infty\cdot \Psi_0^{-1}$. Consider
$$\underline{v}:=\underline{v}^{(0)}\cdot \Psi_\infty=\underline{v}'\cdot \Psi_0.$$
As in the proof of lemma \ref{lemPoleOrder-n}, we conclude
from
$$
\cL\subset\cV^{\alpha_1}\subset\cV^{>\alpha_\mu-1-n}=z^{-n}\cV^{>\alpha_\mu-1} \subset z^{-n}\cO_{\dC\times \cbl}\otimes_{\cO_\dC}\cH
$$
that $\cL \subset z^{-n}\cO_{\dC\times \cbl}\otimes_{\cO_\dC}\cH$. It follows that the matrix
$\Psi_\infty$ satisfies
$\Psi_\infty=\de_\mu+\sum_{k=1}^{n} z^{-k} C_{k}.$
Uniqueness is now also clear.
\item
This follows from $z^{-w}P(\cL,\cL) = \cR$ and
$z^{-w}P(v_i,v_j)-z^{-w}P(v_i^{(0)},v_j^{(0)})\in z^{-1} \cO_{U_2}[z^{-1}]$.
\item
First we introduce two notations: within $\cV^{>-\infty}$ and
$\cV^\alpha$ we consider the $\pi^{-1}\cO_\cbl$-module $\cC^\alpha$
consisting of elementary sections of order $\alpha$ on $\dC^*\times \cbl$.
Then any $v_i\in \cL_{|U_1\times U_2}$ can be written as a sum
$v_i = \sum_{\beta\geq \alpha_1} s(v_i,\beta)$ where
$s(v_i,\beta)\in \cC^\beta(\dC\times U_2)$.

That $\underline{v}^{(0)}$ is a good basis means that
$v_i^{(0)}=\sum_{\beta\geq\alpha_i} s(v_i^{(0)},\beta)$ and that
$$
\Gr^V_\alpha\cH=\bigoplus_{i,k:k\geq 0,\alpha_i+k=\alpha}
\cO_\dC\cdot z^k\cdot s(v_i^{(0)},\alpha_i).
$$
Then for any $\beta$ all sections $z^k\cdot s(v^{(0)}_i,\alpha_i)$
where $k\in \dZ,\alpha_i+k=\beta$ are linearly independent.
In a small neighborhood $U_3\subset U_2$ of $\cH$ the sections
$(z^k\cdot s(v_i,\alpha_i))_{k\in\dZ,\alpha_i+k=\beta}$ inherit this
property of being linearly independent.

For any $v_j$, define $\beta_j$ to be the unique complex number such
that $s(v_j,\beta_j)\neq 0$ and $s(v_j,\beta)=0$
for $\beta<\beta_j$. We have $\beta_j\leq \alpha_j$.
Formula \eqref{v=v_0} and the linear independence of the
$z^k\cdot s(v_i^{(0)},\alpha_i)$ show
\begin{eqnarray}\label{v=v_02}
s(v_j,\beta_j)= \delta_{\beta_j,\alpha_j}\cdot s(v_j^{(0)},\alpha_j)
+ \sum_{k,i,j:\alpha_i-k=\beta_j} (C_k)_{ij}\cdot z^{-k}\cdot s(v_i^{(0)},\alpha_i)
\end{eqnarray}
(remember that this is an equation in $\cV^{>-\infty}$, where the sections $v_j$ of $\cH\subset V^{>-\infty}$
has been extended to sections in $\pi^{-1}V^{>-\infty} \subset \cV^{>-\infty}$).

The main point is to show
\begin{eqnarray}\label{v=v_03}
(C_k)_{ij}=0 \qquad \textup{ for }\alpha_i-k<\alpha_j.
\end{eqnarray}
Then $\beta_j=\alpha_j$ and $s(v_j,\alpha_j)\in \Gr^V_{\alpha_j} \cL(U_2)$.
From this and the linear independence of the $z^k\cdot s(v_i,\alpha_i)$
it follows that $\underline{v}$ is a good basis, first on a
small $U_3\subset U_2$, then on all of $U_2$.

In order to show \eqref{v=v_03} we argue indirectly.
Suppose $(C_k)_{ij}\neq 0$ for some $\alpha_i-k<\alpha_j$ and suppose that
$\alpha_i-k$ is minimal with this property.
Then $\beta_j=\alpha_i-k$ for this $j$, and $\beta_l=\alpha_l$ for
all $l$ with $\alpha_l\leq \alpha_i-k$.
Then in a neighborhood $U_4\subset U_2$ of $\cH$ for any
$\gamma\leq \beta_j=\alpha_i-k$
$$
\bigoplus_{l,m:m\geq 0,\alpha_l+m=\gamma} \cO_{\dC\times U_4}\cdot z^m \cdot s(v_l,\alpha_l)
$$
is a submodule of $\Gr^V_\gamma \cL_{|U_4}$ of the same rank and thus
coincides with $\Gr^V_\gamma \cL_{|U_4}$.
But in the case $\gamma=\beta_j=\alpha_i-k$ we have additionally
$s(v_j,\beta_j)\in \Gr^V_\gamma \cL_{|U_4}$, and by \eqref{v=v_02}
and by the linear independence, it is not a linear combination of the
sections above. Thus $(C_k)_{ij}=0$ if $\alpha_i-k<\alpha_j$.
\end{enumerate}
\end{proof}

Next we give a concrete description of the tangent bundle of $\cbl$ using the universal bundle $\cL$.
For this purpose, we will introduce some auxiliary holomorphic bundles on $\cbl$. We will describe
local bases of these bundles, these will be written as row vectors. We
use the convention that given a (sheaf of) $\cA$-module(s) $\cN$, one can multiply matrices
with entries of $\cN$ with matrices with entries in $\cA$ by scalar multiplication.
Moreover, we make use of the tensor product of matrices,
and in particular of the following rules, where the matrices involved
are supposed to have the appropriate size.
\begin{eqnarray}
(A\otimes B)\cdot(C\otimes D)=(A\cdot C)\otimes (B \cdot D), \\
(A\otimes B)^{-1}=A^{-1}\otimes B^{-1}, \\
(A\otimes B)^{tr}=A^{tr}\otimes B^{tr}, \\
\label{eqMatVecTens} (X\otimes Y) \cdot A^{vec} = (Y \cdot A \cdot X^{tr})^{vec}.
\end{eqnarray}
In the last formula, we denote for any matrix $A\in M(m\times n, \cA)$
by $A^{vec}\in M(nm\times 1,\cA)$ the column vector obtained
by stacking the columns of $A$ in a single one.
Finally, we denote the sheaf of rings $\cO_{\dC\times\cbl}$ by $\cR$ and
its localization along $\{0\}\times\cbl$ by $\cR[z^{-1}]$.

Define $\cM:={\cE\!}nd_{\cR[z^{-1}]}(\cV^{>-\infty})$,
then $\cM$ is a meromorphic bundle with connection induced from $\cV^{>-\infty}$.
This connection is obviously regular singular so that $\cM$ carries its own
$V$-filtration, characterized by $\cV^0\cM=\left\{\phi\in\cM\,|\,\phi(\cV^\alpha)\subset\cV^\alpha\,;\;\forall\alpha\right\}$.
Consider the $\cR$-submodule $\GIw:={\cH\!}om_\cR(\cL,z^{-n}\cL)$
of $\cM$, and the quotient
$$
\GI:=\frac{\GIw}{{\cE\!}nd_\cR(\cL)}
$$
$\GI$ is a $z$-torsion sheaf and can be identified with ${\cH}\!om_\cR(\cL,z^{-n}\cL/\cL)$. As an
$\cO_\cbl$-module, it is locally free of rank $n\mu^2$.
Any section $v$ of the projection $\cL\twoheadrightarrow k_*(\cL/z\cL)$
(here $k:\cbl\hookrightarrow \dC\times\cbl, t\mapsto(0,t)$) yields an isomorphism
$$
\GI \cong\left[{\cE}\!nd_{\cO_\cbl}(\cL^{sp,v})\right]^n
$$
where $\cL^{sp,v}:=\mathit{Im}(v)$. Any local basis $\underline{v}=(v_1,\ldots,v_\mu)$ of $\cL$
in a neighborhood of a point $(0,t)\in\dC\times\cbl$ (in particular, this gives
a section $v:k_*(\cL/z\cL)\rightarrow \cL$ locally) yields a local basis of $\GI$, namely
$(\underline{z}\otimes\underline{v}^*\otimes\underline{v})\in M(1\times n\mu^2,\GI)$,
where the symbol $\underline{z}$ denotes the vector $(z^{-1},\ldots,z^{-n})$.
A local section $\phi\in\GI$ is written in this basis as
\begin{equation}
\label{eqIndBasis}
\phi=\sum_{
\begin{array}{c}
\scriptstyle k=1,\ldots,n\\
\scriptstyle i,j=1,\ldots,\mu
\end{array}
}(\Delta_k)_{ji} z^{-k}\otimes v^*_i\otimes v_j
=(\underline{z}\otimes \underline{v}^* \otimes \underline{v})(\sum_{k=1}^n e_k \otimes \Delta^{vec}_k)
\end{equation}
where $\Delta_k\in M(\mu\times\mu, \cO_\cbl)$ and $e_k\in M(n\times 1, \dC)$ is the $k$-th unit vector.

We will define a chain $\GIV\subset\GIII\subset\GII\subset\GI$ of subbundles
of $\GI$, and an injective morphism $\cT_\cbl\hookrightarrow\GI$ with image
equal to $\GIV$. This will give the description of the tangent bundle alluded to above.
\begin{definition}
Put
$$
\begin{array}{rcl}
\GIIw & := & \left\{\phi\in\GIw\,|\,P(\phi(a),b)+P(a,\phi(b))\in z^w\cR\quad\forall\,a,b\in\cL\right\},\\ \\
\GIIIw & := & \GIIw\cap\cV^0\cM, \\ \\
\GIVw & := & \left\{\phi\in\GIIIw\,|\,\mathit{ad}(z^2\nabla_z)(\phi)=[z^2\nabla_z,\phi]\in{\cE}\!nd_\cR(\cL)\right\},
\end{array}
$$
and define $\GII, \GIII$ resp. $\GIV$ to be the images of $\GIIw, \GIIIw$ resp. $\GIVw$ in $\GI$.
\end{definition}

From the definition it is clear that $\GII, \GIII$ and $\GIV$ are
$\cO_\cbl$-coherent. The following result yields local bases for $\GII$ and $\GIII$
showing that they are in fact locally free. The same is true for $\GIV$, but it is more complicated
to give an explicit local base for that bundle. Instead, we give a characterization
of the elements of $\GIV$. Its local freeness will be shown in lemma \ref{lemKS}.
\begin{lemma}
Let $\underline{v}$ be a local basis of $\cL$ as above and suppose moreover
that $P^{mat}:=(P(v_i,v_j))=(\delta_{i+j,\mu+1})$. Then we have
$$
\GII\cong \bigoplus_{(k,i,j)\in N}
\cO_\cbl z^{-k}\otimes \left(v_i^*\otimes v_j+(-1)^{k+1} v_{\mu+1-j}^*\otimes v_{\mu+1-i}\right),
$$
where
$$
N:=\left\{(k,i,j)\in\{1,\ldots,n\}\times\{1,\ldots,\mu\}^2\,|\,i+j<\mu+1 \textup{ if } k \textup{ is even, }
i+j\leq\mu+1\textup{ if }k\textup{ is odd } \right\}.
$$
Suppose moreover that $\underline{v}$ is a good basis in the sense of \cite{SM},
i.e., that $\underline{v}$ induces a basis
of $\Gr^\bullet_\cV(\cL/z\cL)$. Order the elements of $\underline{v}$ in such a way
that $v_i\in\Gr^{\alpha_i}_\cV(\cL/z\cL)$. Then
$$
\GIII\cong \bigoplus_{
\begin{array}{c}
\scriptscriptstyle (k,i,j)\in N \\
\scriptscriptstyle \alpha_i-k\geq\alpha_j
\end{array}
} \cO_\cbl z^{-k} \left(v_i^*\otimes v_j+(-1)^{k+1} v_{\mu+1-j}^*\otimes v_{\mu+1-i}\right).
$$
Although there is no simple choice for a basis of $\GIV$, its elements
can be characterized as follows: An endomorphism $\phi \in \GIIIw$
lies in $\GIVw$ iff the $\cR[\epsilon]/(\epsilon^2)$-module
$$
\widetilde{\cL}:=\bigoplus_{i=1}^\mu\cR[\varepsilon]/(\varepsilon^2)
\left(v_i+\varepsilon \phi(v_i)\right),
$$
is stable under $z^2\nabla_z$.
\end{lemma}
\begin{proof}
The first point follows from the simple computation
$$
P(\phi(\underline{v})^{tr},\underline{v})+P(\underline{v}^{tr},\phi(\underline{v}))=
\underline{v}\cdot\sum_{k=1}^n z^{-k}\left(\Delta_k^{tr}\cdot P^{mat} + (-1)^k P^{mat} \cdot \Delta_k\right)
=0\quad\quad\textup{mod }\cL
$$
From the condition $\Delta_k^{tr}\cdot P^{mat} + (-1)^k P^{mat} \cdot \Delta_k=0$ one easily deduces the above bases of
$\GII$ and $\GIII$.
For the description of $\GIV$, note that $\widetilde{\cL}$ is stable under
$z^2\nabla_z$ iff there is $B' \in M(\mu\times\mu,\cO_\cbl)$ such that
$$
(z^2\nabla_z)(\underline{v}+\varepsilon\phi(\underline{v}))=\underline{v}\cdot(B+\varepsilon B'),
$$
where $(z^2\nabla_z)(\underline{v})=\underline{v}\cdot B$. In the ring $\cR[\varepsilon]/(\varepsilon^2)$,
this is equivalent to
$$
[z^2\nabla_z,\phi](\underline{v})=\underline{v}\cdot B',
$$
i.e., to $\phi\in\GIV$.
\end{proof}
We are now in the position to compare the bundle $\GI$ with the tangent bundle
of $\cbl$. Define the following morphism:
$$
\begin{array}{rcl}
\KS:\cT_\cbl & \longrightarrow & \GI \\ \\
X & \longmapsto & [v\in\cL \mapsto \nabla_X v]
\end{array}
$$
where the brackets on the right-hand side denote the class in
the quotient $\GI$.
\begin{lemma}
\label{lemKS}
$\GIV$ is a bundle, and $\KS$ is a bundle isomorphism from $\cT_\cbl$ to $\GIV$.
\end{lemma}
\begin{proof}
First we will prove that given $X\in\cT_\cbl$, the covariant derivative $\nabla_X$
really defines an element in $\GI$. It was already shown that for any $v\in\cL$,
$\nabla_X v$ lies in $z^{-n}\cL$. On the other hand, if $f\in\cO_\cbl$,
then $\nabla_X(f\cdot v)=f\nabla_X(v)+X(f)\cdot v$, but $X(f)\cdot v\in\cL$, so modulo
${\cE}\!nd(\cL)$ we have $\nabla_X(f\cdot v)=f\nabla_X(v)$. Moreover, the flatness of $P$ implies
that $P(\nabla_X(-),-)+P(-,\nabla_X(-))=XP(-,-)$, so $\mathit{Im}(KS)\subset\GII$
and it is clear that the $\cV$-filtration is respected, i.e., that we
get an element in $\cV^0\cM$, as we derive only in parameter direction.
This proves $\mathit{Im}(\KS)\subset\GIII$. From $[z^2\nabla_z,\nabla_X]=0$ it follows immediately
that $\mathit{Im}(\KS)\subset\GIV$.

The last step of the proof is to show that $\KS$ maps $\cT_\cbl$
isomorphically onto $\GIV$. Then it follows that $\GIV$ is locally free.
We will use the construction of coordinates on the fibers resp. on the base of the projection
$\cbl\to\DcPMHS$ in \cite[Ch. 5]{He2} resp. in  \cite[proof of theorem 12.8]{He3}.
We will rephrase the outcome of these constructions,
without going into details. From that it will follow that
$\KS$ is a bundle isomorphism onto $\GIV$.

Having fixed a reference TERP-structure, we first consider the larger spaces
\begin{eqnarray*}
D_{Sp}&=& \{\cH\subset V^{\alpha_1}\ |\ \cH \textup{ is a free }
\cO_\dC-\textup{module} \\
&& \hspace*{2.2cm} \textup{with spectral numbers } \alpha_1,...,\alpha_\mu\},\\
D_{Fl}&=& \{\widetilde{F}^\bullet_\cH  |\ \cH\in D_{Sp}\},
\end{eqnarray*}
where $\widetilde{F}^\bullet_\cH $ denotes the filtration
defined by $\cH$ on the space $H^\infty$ (see definition
\ref{defHodgeFiltSpectrum}).
We have an obvious projection $D_{Sp}\to D_{Fl}$. Here $D_{Fl}$ is a flag manifold,
and $D_{Sp}$ is also a manifold, with fibers isomorphic to
some $N_{Sp}\in\dN$. Local coordinates on $D_{Sp}$ and $D_{Fl}$
can be chosen as follows:
For some $\cH\in D_{Sp}$ one fixes a good basis $\underline{v}^0$.
The analogue of lemma \ref{lemCanExtBase} holds and provides
a unique good basis $\underline{v}$ in a
neighborhood $U$ of $\cH\in D_{Sp}$ where
\begin{eqnarray*}
\underline{v}=\underline{v}^0\cdot
\left(\de_\mu+\sum_{k=1}^n z^{-k} C_k\right)
\end{eqnarray*}
with $(C_k)_{ij}=0$ if $\alpha_i-k<\alpha_j$.
Now the $(C_k)_{ij}$ with $\alpha_i-k\geq\alpha_j$ are local coordinates
on $D_{Sp}$ and those with $\alpha_i-k=\alpha_j$ are local coordinates
on $D_{Fl}$.

The construction of coordinates on $\cbl$ in \cite{He2}\cite{He3}
amounts to the following:
For $\cH\in \cbl\subset D_{Sp}$ a very special good basis was chosen.
It gives local coordinates $(C_k)_{ij}$ with $\alpha_i-k\geq\alpha_j$
on $D_{Sp}$. The conditions from the pairing $P$ and the pole of order
$2$, which determine $\cbl$ in $D_{Sp}$ locally, were shown to give a set of equations in $(C_k)_{ij}$
whose linear parts are linearly independent. This proved the smoothness
of $\cbl$.

Now the definition of $\GIV$ uses exactly these linear parts.
This shows the injectivity of $\KS$ and that $\mathit{Im}(\KS)=\GIV$.
\end{proof}
By abuse of notation, we call $\KS$ the Kodaira-Spencer morphism, although this is the
correct name only if we consider $\KS$ as an isomorphism between
$\cT_\cbl$ and $\GIV$.

It will be useful to have a local characterization of $\GIV$ in the basis
of $\GIII$ given above. Let $\phi=
(\underline{z}\otimes \underline{v}^* \otimes \underline{v})(\sum_{k=1}^n e_k \otimes \Delta^{vec}_k)$
be a local section of $\GIII$, i.e., $(\Delta_k)_{ij}+(-1)^k(\Delta_k)_{\mu+1-j,\mu+1-i}=0$ and
$(\Delta)_{ij}=0$ for all $\alpha_i-k<\alpha_j$. Then $\Delta\in\GIV$ iff there is
$B'\in M(\mu\times\mu, \cR)$ such that
$$
(z^2\nabla_z)(\underline{v}+\varepsilon \phi(\underline{v})) = \underline{v}\cdot(B+\varepsilon B')
$$
where $B\in M(\mu\times\mu, \cR)$ is defined by $(z^2\nabla_z)(\underline{v})=\underline{v}\cdot B$.
This is equivalent to
$$
B'=[B,\sum_{k=1}^n z^{-k}\cdot \Delta_k]
+\sum_{k=1}^n (-k)z^{-k+1}\cdot \Delta_k.
$$
$B'$ is required to be holomorphic, so that the coefficients
of all strictly negative powers of $z$ in this equation must vanish.
Writing $B=\sum_{k=0}^\infty z^k\cdot B_{k-1}$, there are a priori conditions
for any $l=1,...,n$ (with $\Delta_{n+1}:=0$):
\begin{eqnarray}
\label{eqCondForGIV}
0=(\textup{coefficient of }z^{-l})
=(-1-l)\Delta_{l+1}+ \sum_{k=l}^{n} [B_{k-1-l},\Delta_{k}].
\end{eqnarray}
However, as we are working with a good basis $v$ of $\cL$,
it follows that $v_i\in\cV^{\alpha_i}$ and $(z\nabla_z)v_i\in\cV^{\alpha_i}$. This gives
$(B_k)_{ij}=0$ for all $\alpha_i+k<\alpha_j$. Remember also that $(\Delta_k)_{ij}=0$ for
$\alpha_i-k<\alpha_j$. The equation $[\Delta_n, B_{-1}]=0$ is thus trivially
satisfied, so that the conditions \eqref{eqCondForGIV} are non-empty only for $l\in\{1,\ldots,n-1\}$.

We will now define the analogue of the subbundle of horizontal tangent directions
on the classifying spaces $\DcPHS$ in the sense of \cite{Sch} for the space $\cbl$.
\begin{deflemma}\label{defHorBundles}
Define the following subsheaf of $\GI$:
$$
\GIh:=\textup{Image of }{\cH}\!om_\cR(\cL,z^{-1}\cL)\textup{ in }\GI,
$$
and put $\GIIh:=\GIh\cap \GII$, $\GIIIh:=\GIh\cap \GIII$ and
$\GIVh:=\GIh\cap \GIV$. Then $\GIh, \GIIh$
and $\GIIIh$ are $\cO_\cbl$-locally free.
$\GIVh$ is an $\cO_\cbl$-coherent subsheaf of $\GIV$. It is
equal to $\GIV$ and thus locally free if $n=\lfloor\alpha_\mu-\alpha_1\rfloor=1$.
We call $\cT^{hor}_\cbl:=\KS^{-1}(\GIVh)$ the subsheaf of horizontal
tangent directions or horizontal tangent sheaf for short.
\end{deflemma}
\begin{proof}
The $\cO_\cbl$-coherence of all of the sheaves in question is obvious from their
definition, and one obtains local bases of $\GIh$, $\GIIh$ resp. $\GIIIh$ by restricting
the bases of $\GI$, $\GII$ resp. $\GIII$ to $k=1$. A section $\phi=
(\underline{z}\otimes \underline{v}^* \otimes \underline{v})(\sum_{k=1}^n e_k \otimes \Delta^{vec}_k)$
is contained in one of these subbundles iff $\Delta_k=0$ for $k>1$.
The equality $\GIVh=\GIV$ for $n=1$ is obvious from the definition
as we have $\GIh=\GI$ in this case.
\end{proof}

In the remainder of this section, we show that $\GIVh$ is not locally
free in general. For that purpose, consider the local basis of
$\GIIIh$ given above. In this basis, the conditions for a section
$\phi\in \GIIIh$ to be an element in $\GIVh$ are simply obtained from
formula \eqref{eqCondForGIV} by putting all $\Delta_k=0$ if $k>1$. This yields
the unique equation
\begin{equation}
[B_{-1},\Delta_1]=0.
\end{equation}
Note that $B_{-1}$ is the matrix of the pole part of $\nabla_z$ with respect to $\underline{v}$, i.e.,
of the endomorphism $\cU:=[z^2\nabla_z]\in{\cE}\!nd_{\cO_\cbl}(\cL/z\cL)$. Similarly, $\Delta_1$ is
by definition the matrix of the class $[z\phi]\in {\cE}\!nd_{\cO_\cbl}(\cL/z\cL)$. This shows that we
have the following simple characterization
$$
\GIVh=\left\{\phi\in\GIIIh\,|\,[\cU,[z\phi]]=0\right\}.
$$
We have $\cU,[z\phi]:\cV^\bullet(\cL/z\cL)\subset\cV^{\bullet+1}(\cL/z\cL)$, so
that $[\cU,[z\phi]]=0$ if $n=1$. This implies $\GIVh=\GIIIh$, in particular,
$\GIV=\GIIIh$ in this case.

If $n\geq 2$ then in general $\GIVh$ will not be locally free. The
reason is that in general the rank of the condition $[\cU,[z\phi]]=0$
varies within $\cbl$. Note however that the base $\DcPMHS$ carries a horizontal
subbundle, as it is a homogeneous space, so that the obstruction for
$\GIVh$ to be locally free lies in the fibers of
$\cbl\to\DcPMHS$. We will describe a situation where this actually occurs.

For simplicity we restrict to a situation in which $\dim C^\alpha=1$ for all $\alpha$.
Then $N=0$ and $\DPHS=\DcPHS=\DPMHS=\DcPMHS=\{pt\}$, so that
$D_\mathit{BL}=\cbl=\dC^{N_\mathit{BL}}$ for
some $N_\mathit{BL}$.
We choose $v_i^0\in C^{\alpha_i}$ with $P(v_i^0,v_{\mu+1-j}^0)=\delta_{ij}$.
By \cite[Ch. 5]{He2} formula \eqref{v=v_0} holds on all of $\cbl$, with
\begin{eqnarray*}
(C_k)_{ij}&=&0 \qquad \textup{ for }\alpha_i-k\leq \alpha_j,\\
(\alpha_i-k-\alpha_j)\cdot(C_k)_{ij}&=& \sum_{l}(\alpha_l-1-\alpha_j)\cdot
(C_{k-1})_{il}(C_1)_{lj},\\
(C_1)_{ij}&=&(C_1)_{\mu+1-j,\mu+1-i},\\
\cU([\underline{v}]) &=& [\underline{v}]\cdot
((\alpha_i-1-\alpha_j)\cdot (C_1)_{ij}),
\end{eqnarray*}
and global coordinates on $\cbl$ are given by
those $(C_1)_{ij}$ where $i+j\leq\mu+1$.
The zero point of these coordinates is the TERP-structure $\cH^0=\bigoplus\cO_\dC\cdot v_i^0$.

In the basis of $\GIIIh$ in definition-lemma \ref{defHorBundles}
$$[z\phi]([\underline{v}])=[\underline{v}]\cdot \Delta_1$$
with $(\Delta_1)_{ij}=0$ if $\alpha_i-1\leq \alpha_j$
(equality is impossible due to $\dim C^\alpha=1$)
and $(\Delta_1)_{ij}=(\Delta_1)_{\mu+1-j,\mu+1-i}$.
We have $\rank \GIIIh=N_{BL}$ and
$$
\cT^{hor}_{|0}\cbl \cong \cG^\mathit{IV,h}_{|0} = \cG^\mathit{III,h}_{|0} \cong \cT_{|0}\cbl
$$
as $\cU$ and $[z\phi]$ commute at the point $0$. To prove that $\GIVh$
is not locally free it is sufficient to show that
the condition $[\cU,[z\phi]]=0$ is non-empty at some point $t\in \cbl$ .

If there are $\alpha_i,\alpha_l,\alpha_m$ with
$\alpha_i-2>\alpha_l-1>\alpha_m$ and $m\neq \mu+1-i$ then
the $(i,m)$ entry of the commutator of the matrices corresponding
to $\cU$ and $[z\phi]$ is
$$
\sum_j \left((\alpha_i-1-\alpha_j)(C_1)_{ij}\cdot (\Delta_1)_{jm}
- (\Delta_1)_{ij}\cdot (\alpha_j-1-\alpha_m)(C_1)_{jm}\right).
$$
This sum is non-empty because $j=l$ gives a term,
and all present $(\Delta_1)$-coefficients are different.
So if $(C_1)_{il}(t) \neq 0$ this gives a non-empty condition, and
$\rank \cT^{hor}_{|t}\cbl < \rank \cT^{hor}_{|0}\cbl$.

An example of this type can be constructed starting with a suitable
semiquasihomogeneous deformation of a Brieskorn-Pham singularity
$f=x_0^{a_0}+x_1^{a_1}+x_2^{a_2}$ where $\textup{gcd}(a_i,a_j)=1$ for
$i\neq j$ and such that $\frac{1}{a_0}+\frac{1}{a_1}+\frac{1}{a_2}$ is
sufficiently
small.

\section{Holomorphic sectional curvature}\label{secHolSectCurv}
\setcounter{equation}{0}

One of the most interesting features of a TERP-structure
is the construction of a canonical extension to a twistor, i.e.,
a $\dP^1$-bundle. Starting with a family of TERP-structures,
this yields a $\cC^\infty$-family (which is actually real analytic) of twistors. Let us briefly recall how this is done
(see \cite{He4} and \cite{HS1} for more details).

Given a TERP-structure $\TERP$, define for any $z\in\dC^*$ the anti-linear involution
$$
\begin{array}{rcl}
\tau: H_z & \longrightarrow & H_{1/\overline{z}}\\
s & \longmapsto & \nabla\textup{-parallel transport  of }\overline{z^{-w}s}
\end{array}
$$
The image $\tau(\cH_0)$ of the germ of sections of $\cH$ at zero is contained in the germ
$(\widetilde{i}(\cH'))_\infty$, where $\widetilde{i}:\dC^*\rightarrow \dP^1\backslash\{0\}$.
This defines an extension of $H$ to infinity, which is a holomorphic $\dP^1$-bundle, i.e., a twistor.
We will denote it by $\widehat{H}$. If $\widehat{H}$ is trivial, then we call the original
TERP-structure pure. Moreover, in this case we can define a hermitian pairing
$h$ on $H^0(\dP^1,\widehat{\cH})$ by the formula $h(a,b):=z^{-w}P(a,\tau b)$. If this
form is positive definite, then $\TERP$ is called pure polarized.
We only remark (this is discussed in detail in \cite{He4} and \cite{HS1}) that
if we start with a family of TERP-structures, then this extension procedure yields
a real analytic family of twistors.
Define
$$
\cblp:= \left\{\cH \in \cbl\,|\, \widehat{\cH}\mbox{ is pure polarized }\right\}.
$$
One of the main results in \cite{HS1}, namely, theorem 6.6 says
that a regular singular TERP-structure is an element of $D_\mathit{BL}$ iff it induces a
nilpotent orbit, i.e. iff the pullback
$$
\pi_r^*(H,H'_\dR,\nabla,P),
$$
where  $\pi_r:\dC\to\dC,\;z\mapsto r\cdot z$,
is a pure polarized TERP-structure
for any $r\in \dC^*$ with $|r|$ sufficiently small.

Such a pullback is then also an element of $D_\mathit{BL}$. Therefore the set
$\cblp$ of all pure polarized TERP-structures in $\cbl$
is non-empty, and it intersects $D_\mathit{BL}$ nontrivially.
The condition to be pure and polarized
is open, so $\cblp$ is an open submanifold of $\cbl$.

There is no reason to expect $\cblp\subset D_\mathit{BL}$, but the intersection
$\cblp\cap D_\mathit{BL}$ seems to be most interesting for applications.
If $N=0$ then $\DPHS=\DPMHS$ and moreover
$\cblp\cap D_\mathit{BL}$ contains a neighborhood of the
zero section $\DPMHS\hookrightarrow D_\mathit{BL}$. The reason is that if $N=0$,
then the action by pullback $\pi_r^*$ coincides with the good $\dC^*$-action
considered in \cite[Theorem 5.6]{He2}.

Performing the above construction on the whole classifying space $\cbl$ yields
an extension of the universal bundle $\cL$ to a real analytic family of twistors
$\widehat{\cL}$, that is, a locally free $\cO_{\dP^1}\cC^{an}_\cbl$-module.
Moreover, on the subspace $\cblp$ the sheaf of fibrewise global sections
$p_*\widehat{\cL}_{|\cblp}$ is by definition locally free over $\cC^{an}_\cblp$
and comes equipped with a positive definite hermitian metric $h$. We will show
that this induces positive definite hermitian metrics on the bundles $\GI,\ldots,\GIV$, restricted
to $\cblp$.

Denote by $\cK$ the sheaf $\CanD\otimes\cL_{|\cblp}$ and put $\Ks:=p_*\widehat{\cL}_{|\cblp}$,
then we have a hermitian metric $h:=z^{-w}P(-,\tau-)$ on
the $\CanD$-module $\Ks$. We obtain a splitting
$$
k^{-1}\cK = \cK^{sp} \oplus k^{-1}(z\cK)
$$
where $k:\cblp\hookrightarrow\dC\times\cblp, t\mapsto (0,t)$. This yields
$$
\begin{array}{c}
\CanD\otimes k^{-1}\GI =
{\cH}\!om_{k^{-1}\cO_\dC\CanD}\left(\cK^{sp} \oplus k^{-1}(z\cK), k^{-1}\left(\frac{z^{-n}\cK}{\cK}\right)\right) \cong \\ \\
{\cH}\!om_{\CanD}\left(\Ks, \oplus_{k=1}^n z^{-k}\Ks\right) = \oplus_{k=1}^n {\cH}\!om_{\CanD}\left(\Ks, z^{-k}\Ks\right)
\cong \left[{\cE}\!nd_{\CanD}(\Ks)\right]^n
\end{array}
$$
We obtain a hermitian metric on ${\cE}\!nd_{\CanD}(\Ks)$ (and its powers) by
$h(\phi,\phi')=Tr(\phi\cdot(\phi')^*)$, where $(-)^*$ denotes the hermitian adjoint.
This induces a metric on $\GI$ and by restriction on the subbundles $\GII$, $\GIII$ and $\GIV$.
We denote these metrics by $h^{I}, \ldots, h^{IV}$.
We remark that choosing any local basis $\underline{u}\in
M(1\times n, \cK)$ of $\cK$ in a neighborhood of a point
$(0,t)\in\{0\}\times\cblp$ yields a similar splitting
$$
k^{-1}\cK=\cK^{\mathit{sp},\underline{u}}\oplus k^{-1}(z \cK)\mbox{ and }
\CanD\otimes \GI\cong \left[{\cE}\!nd_{\CanD}(\cK^{\mathit{sp},\underline{u}})\right]^n,
$$
where $\cK^{\mathit{sp},\underline{u}}:=\oplus_{i=1}^\mu \CanD u_i$.
If $\underline{u}$ is a global basis of $\widehat{\cL}_{|\cblp}$, then
$\cK^{\mathit{sp},\underline{u}}=\Ks$. If $\underline{u}$ happens to be holomorphic,
i.e., $\underline{u}\in M(1\times \mu, \cL)$, then $\GI\cong
\left[{\cE}\!nd_{\cO_\cblp}(\cL^{\mathit{sp},\underline{u}})\right]^n$, this isomorphism was
already considered in section \ref{secTangent}.
Similarly to the holomorphic basis from
formula \eqref{eqIndBasis}, we obtain a basis of $\CanD\otimes \GI\cong \left[{\cE}\!nd_{\CanD}(\cK^{\mathit{sp},\underline{u}})\right]^n$,
namely, $\underline{z}\otimes\underline{u}^*\otimes
\underline{u}\in M(1\times n\mu^2,\CanD\otimes \GI)$ and
any section $\phi$ of
$\CanD\otimes \GI$ is written in the basis $\underline{z}\otimes\underline{u}^*\otimes
\underline{u}$ as
$$
\phi=\sum_{
\begin{array}{c}
\scriptstyle k=1,\ldots,n\\
\scriptstyle i,j=1,\ldots,\mu
\end{array}
}(\Gamma_k)_{ji} z^{-k}\otimes u^*_i \otimes u_j
=(\underline{z}\otimes \underline{u}^* \otimes \underline{u})(\sum_{k=1}^n e_k \otimes \Gamma^{vec}_k)
$$
(remember that $\underline{z}:=(z^{-1},\ldots,z^{-n})$, that
$e_k\in M(n\times 1, \dC)$ is the $k$-th unit vector and that $A^{vec}$ denotes the column vector of a
matrix $A$ as explained after formula \eqref{eqMatVecTens}).

Recall the Kodaira-Spencer map from lemma \ref{lemKS}
$$
\begin{array}{rcl}
\KS:\cT_\cblp & \hookrightarrow & \GI \\ \\
X & \longmapsto & [v\mapsto \nabla_X v]),
\end{array}
$$
which endows $\cT_\cblp$ with a positive definite hermitian metric which we simply denote by $h$.
Recall also that we denoted by $\cT^{hor}_\cblp$ the coherent subsheaf of $\cT_\cblp$ defined by
$\cT^{hor}_\cblp:=\KS^{-1}({\cH}\!om_{\cO_{\dC\times\cblp}}(\cH,\frac{z^{-1}\cH}{\cH}))$,
and that it is not locally free in general. However, it contains the zero section of
$T_\cblp\rightarrow \cblp$, and we may consider the holomorphic sectional
curvature of the metric $h$ on vectors of $T_\cblp^{hor}\backslash\{\textup{zero section}\}$. Let us briefly recall its the definition:
Given any holomorphic bundle $E$ on a complex manifold $M$ and a positive definite hermitian metric
$h$ on $E$, there is a unique connection $D:\cE\rightarrow \cE\otimes\cA^1_M$ such that
$D(h)=0$ and such that the $(0,1)$-part of $D$ coincides with the operator defining
the holomorphic structure of $\cE$. $D$ is called the Chern connection of $(E,h)$. Its
curvature is by definition the section $R$ of ${\cE}\!nd_{\CanM}(\CanM\otimes\cE)\otimes \cA^{1,1}_M$ given
by $e\stackrel{R}{\mapsto} D^{(2)}(D(e))$, here $D^{(2)}:\cE\otimes\cA^1_M\rightarrow\cE\otimes\cA^2_M$, $D^{(2)}(e\otimes\alpha)
=D(e)\wedge\alpha+s\otimes d\alpha$. If $E$ is the holomorphic
tangent bundle of $M$, then the function
$$
\begin{array}{rcl}
\kappa: T_M \backslash \{\mbox{zero section}\} & \longrightarrow & \dR \\ \\
\xi&\longmapsto& h(R(\xi,\overline{\xi})\xi,\xi)/h(\xi,\xi)^2
\end{array}
$$
is called the holomorphic sectional curvature of $M$.

We are now able to state the main result of this section.
\begin{theorem}
The restriction of the holomorphic sectional curvature
$\kappa:T_\cblp\backslash\{\mbox{zero-section}\} \rightarrow \dR$
to the (linear space associated to the) coherent subsheaf $\cT_\cblp^{hor}$ is
bounded from above by a negative real number.
\end{theorem}
\begin{proof}
First recall a formula for the curvature tensor of the Chern connection on an arbitrary bundle.
Let, as before, $E$ be a holomorphic vector bundle of rank $\mu$ on a complex manifold $M$ and $h:\cE\otimes\overline{\cE}
\rightarrow\CanM$ a positive definite hermitian metric. For a local holomorphic basis
$\underline{e}\in M(1\times\mu,\cE)$, put $H:=
\left(h(\underline{e}^{tr},\underline{e})\right)^{tr}\in M(\mu\times\mu,\CanM)$.
The curvature $R$ is linear, thus $R(\underline{e})=\underline{e}M_R$, where
$M_R\in M(\mu\times\mu,\cA^{1,1}_M)$. It is well known (see, e.g., \cite[lemma 11.4]{CarlStachPeters})
that
$$
M_R=H^{-1}\overline{\partial}\partial H-H^{-1}\overline{\partial}(H) \wedge H^{-1}\partial H.
$$
In particular, for any holomorphic vector field $X\in\cT_M$, we have
$$
M_R(X,\overline{X})=
-H^{-1}(\overline{X}X)(H)+H^{-1}\overline{X}(H) H^{-1}X(H) \in M(\mu\times\mu,\CanM).
$$
If at a point $x\in M$, $H(x)=Id$, then $M_R(X,\overline{X})(x)=\overline{X}(H)(x)X(H)(x)-(\overline{X}X)(H)(x)$,
or, if we write $\xi:=X(x)$, then
\begin{equation}
\label{eqCurv}
M_R(X,\overline{X})(x)=\overline{\xi}(H)\xi(H)-(\overline{X}X)(H)(x).
\end{equation}
This formula will allow a very significant simplification of the calculations.
Let $t\in\cblp$, and let $\xi\in T^{hor}_t(\cblp)$ be any vector with $\xi\neq 0$. Choose
local holomorphic coordinates $(t_1,\ldots,t_{\dim(\cblp)})$ centered at $t$
such that $(\partial_{t_1})_{|t}=\xi$. Although we are interested in
the curvature tensor $R^{IV}$ of the bundle $\cG^{IV}_{|\cblp}$ (which is isomorphic to $\cT_\cblp$),
our first aim is to give an expression for the
matrix $M_R(\partial_{t_1},\overline{\partial}_{t_1})(t)$
which represents $R^{I}(\partial_{t_1},\overline{\partial}_{t_1})(x) \in\mathit{{\cE}\!nd}_\dC(\cG^{I}_{|t})$
with respect to a holomorphic basis in a neighborhood of $t$. This
basis is induced from a holomorphic basis of $\cL$ near $(0,t)$,
which is obtained as follows: choose a basis
$\underline{v}^0\in M(1 \times \mu, \cL_{|\dC\times\{t\}})$ of $\cL_{|\dC\times\{t\}}$
such that $P((\underline{v}^0)^{tr},\underline{v}^0)=\de_\mu$ and
$\tau(\underline{v}^0)=\underline{v}^0$. Then define
$$
\underline{v}:=\underline{v}^0\left(\de_\mu+\sum_{k=1}^nz^{-k}C_k\right) \in M(1\times\mu,\cL)
$$
to be the extension provided by lemma \ref{lemCanExtBase}. It still satisfies
$P(\underline{v}^{tr},\underline{v})=\de_\mu$, but not necessarily $\tau(\underline{v})=\underline{v}$.
Write $\KS(\xi)=\sum_{k,i,j}(\Delta_k)_{ji} \, (z^{-k}\otimes (v^{(0)}_i)^*\otimes v^{(0)}_j)$
(i.e. $\Delta_k\in M(\mu\times\mu,\dC)$), then it follows from $\xi\in T_t^{hor}(\cblp)$ that
$(\Delta_k)=0$ for all $k>1$. Moreover,
as $\kappa(\partial_{t_1})=(\underline{v}\mapsto\underline{v}(\sum_{k=1}^n z^{-k} \partial_{t_1}C_k))$,
we conclude that $\xi(C_1)=\Delta_1$ and $\xi(C_k)=\Delta_k = 0$ for $k>1$.
The matrices $H:=[h^I((\underline{z}\otimes\underline{v}^*\otimes\underline{v})^{tr},
\underline{z}\otimes\underline{v}^*\otimes\underline{v})]^{tr}$
and $M(\partial_{t_1},\overline{\partial}_{t_1})$ are elements of
$M(\mu\times\mu,\CanD)$, so that we conclude from formula \eqref{eqCurv} that
$M(\partial_{t_1},\overline{\partial}_{t_1})(x)$ can obtained from the image of $H$ under the
reduction map $M(\mu\times\mu,\CanD)\twoheadrightarrow M(\mu\times\mu,\cQ)$,
where $\cQ:=\CanD/(t_1^2,\overline{t}_1^2,t_j,\overline{t}_j)_{j>1}$. In order to keep
the notation simple, we still call this reduction $H$. Moreover, it is clear that
this reduced matrix $H$
may be calculated from the image of the basis $\underline{v}$ under the map
$\cK \twoheadrightarrow \cK\otimes\widetilde{\cQ}$, where
$\widetilde{\cQ}:=\cO_\dC\CanD/(t_1^2,\overline{t}_1^2,t_j,\overline{t}_j)_{j>1}$. Again we
denote this image by $\underline{v}$. All subsequent calculations
take place in either $\widetilde{\cQ}$ or $\cQ$. In particular, we have $C_1=t_1\Delta_1$ and $C_k=0$ for $k>1$
in $\cQ$. This implies $\underline{v}=\underline{v}^{(0)}(\de_\mu+z^{-1}C_1)$
and $\de_\mu=P(\underline{v}^{tr},\underline{v})=
(\de_\mu+z^{-1}C_1)^{tr}P((\underline{v}^0)^{tr},\underline{v}^0)(\de_\mu-z^{-1}C_1)=
(\de_\mu+z^{-1}(C_1^{tr}-C_1))$ so that $C^{tr}_1=C_1$.

Consider the base change
given by $\underline{w}:=\underline{v}(\de_\mu+\frac12[\overline{C}_1,C_1]+z\overline{C}_1)$.
It follows from
$P(\underline{v}^{tr},\underline{v})=\de_\mu$, $\overline{C}^{tr}_1=\overline{C}_1$ and
$[\overline{C}_1,C_1]^{tr}=-[\overline{C}_1,C_1]$ that $P(\underline{w}^{tr},\underline{w})=\de_\mu$. Moreover, as
$$
\underline{w}=\underline{v}^{(0)}\left(\de_\mu+z\overline{C}_1+z^{-1}C_1+\frac12(\overline{C}_1C_1+C_1\overline{C}_1)\right),
$$
we also have $\tau(\underline{w})=\underline{w}$.
It is a simple calculation to show that the inverse base change is given by
$$
\underline{v}:=\underline{w}\cdot(\de_\mu-\frac12[\overline{C}_1,C_1]-z\overline{C}_1)
$$
We obtain an induced base change on $\GI$, given by
$\left(\de_\mu-\frac12[\overline{C}_1,C_1]-z\overline{C}_1)^{tr}\right)^{-1}
\otimes(\de_\mu-\frac12[\overline{C}_1,C_1]-z\overline{C}_1)$, i.e.:
$$
\begin{array}{c}
\underline{z}\otimes \underline{v}^* \otimes \underline{v} =
(\underline{z}\otimes \underline{w}^* \otimes \underline{w})\cdot
\left[\de_n\otimes
\left((\de_\mu-\frac12[\overline{C_1},C_1]-z\overline{C_1})^{-1}\right)^{tr}\otimes
(\de_\mu-\frac12[\overline{C_1},C_1]-z\overline{C_1})\right] = \\ \\
\underline{z}\otimes \underline{w}^* \otimes \underline{w} \cdot
\left(\underbrace{\de_n\otimes (\de_\mu\otimes\de_\mu-\frac12(\de_\mu\otimes[\overline{C_1},C_1]+[\overline{C_1},C_1]\otimes\de_\mu))
+N_z\otimes(\overline{C_1}\otimes\de_\mu-\de_\mu\otimes \overline{C_1})}_{=:X}\right),
\end{array}
$$
here
$$
N_z=\begin{pmatrix}0&1&\ldots &&0\\0 & 0 & 1 &\ldots&0\\&&\vdots\\0 &0 & \ldots & 0&1\\0&0& & \ldots& 0\end{pmatrix}.
$$
Now we have
$$
\begin{array}{c}
H=[h^I((\underline{z}\otimes\underline{v}^*\otimes\underline{v})^{tr},
\underline{z}\otimes\underline{v}^*\otimes\underline{v})]^{tr}=
[X^{tr}h^I((\underline{z}\otimes\underline{w}^*\otimes\underline{w})^{tr},\underline{z}\otimes\underline{w}^*\otimes\underline{w})\overline{X}]^{tr}
\\ \\=\overline{X}^{tr}X=
\de_n\otimes(\de_\mu\otimes\de_\mu-\de_\mu\otimes[\overline{C_1},C_1]-[\overline{C_1},C_1]\otimes\de_\mu)
+(N_z)^{tr}\otimes(C_1\otimes\de_\mu-\de_\mu\otimes C_1)\\ \\
+N_z\otimes(\overline{C_1}\otimes\de_\mu-\de_\mu\otimes \overline{C_1})
+\de_{n-1}\otimes(C_1\overline{C_1}\otimes\de_\mu+\de_\mu\otimes C_1\overline{C_1}-\overline{C_1}\otimes C_1 -C_1\otimes\overline{C_1})
\end{array}
$$
where $\de_{n-1}=N_z^{tr}\cdot N_z=\textup{diag}(0,1,\ldots,1)$. The next step is to invoke formula \eqref{eqCurv} to obtain the matrix
$M_R(\partial_{t_1},\overline{\partial}_{t_1})(x)$. Using $C_1=t\Delta_1$, we get
$$
\begin{array}{rcl}
\overline{\xi}(H) & = & N_z\otimes(\overline{\Delta}_1\otimes\de_\mu-\de_\mu\otimes\overline{\Delta}_1), \\ \\
\xi(H) & = & N_z^{tr}\otimes(\Delta_1\otimes\de_\mu-\de_\mu\otimes\Delta_1), \\ \\
(\overline{\partial}_{t_1}\partial_{t_1}(H))(x)
& = & -\de_n\otimes(\de_\mu\otimes[\overline{\Delta}_1,\Delta_1]+
[\overline{\Delta}_1,\Delta_1]\otimes\de_\mu)+\\ \\
& &
\de_{n-1}\otimes(\Delta_1\overline{\Delta_1}\otimes\de_\mu+\de_\mu\otimes \Delta_1\overline{\Delta_1}-\overline{\Delta_1}\otimes \Delta_1 -\Delta_1\otimes\overline{\Delta_1}),
\\ \\
M_R(\xi,\overline{\xi}) & = &
\de_n\otimes(\underbrace{\de_\mu\otimes[\overline{\Delta}_1,\Delta_1]+[\overline{\Delta}_1,\Delta_1]\otimes\de_\mu}_{=:S})\\\\
& & -\de_{n-1}\otimes(\Delta_1\overline{\Delta}_1\otimes\de_\mu+\de_\mu\otimes\Delta_1\overline{\Delta}_1-
\overline{\Delta_1}\otimes \Delta_1 - \Delta_1\otimes\overline{\Delta_1}) \\ \\
& & +\de'_{n-1}\otimes(\underbrace{\overline{\Delta}_1\Delta_1\otimes\de_\mu+\de_\mu\otimes\overline{\Delta}_1\Delta_1
-\overline{\Delta_1}\otimes \Delta_1 - \Delta_1\otimes\overline{\Delta_1}}_{=:R}),
\end{array}
$$
where $\de'_{n-1}=N_z\cdot N_z^{tr}=\diag(1,\ldots,1,0)$.
What we are really interested in is to give a closed formula for the expression
$h^I(R(\xi,\overline{\xi})\KS(\xi),\KS(\xi))$. As $\KS(\xi)=\sum_{i,j}(\Delta_1)_{j,i}(z^{-1}\otimes
(v_i^{(0)})^*\otimes v^{(0)}_j)$ and $h^I((\underline{z}\otimes \underline{v}^* \otimes \underline{v})^{tr},
(\underline{z}\otimes \underline{v}^* \otimes \underline{v}))=
\de_n\otimes h((\underline{v}^*)^{tr},\underline{v}^*)\otimes h(\underline{v}^{tr},\underline{v})$,
we obtain that
$$
\begin{array}{c}
h^I(R(\xi,\overline{\xi})\KS(\xi),\KS(\xi)) = \\ \\
h^I\left(((z^{-1}\otimes(\underline{v}^{(0)})^*\otimes\underline{v}^{(0)})(R+S)(\Delta_1)^{vec})^{tr},(z^{-1}\otimes(\underline{v}^{(0)})^*\otimes\underline{v}^{(0)})(\Delta_1)^{vec}\right)=\\ \\
\left(S\cdot(\Delta_1)^{vec}\right)^{tr}(\de_\mu\otimes\de_\mu)(\overline{\Delta}_1)^{vec}
=
\left(([[\overline{\Delta}_1,\Delta_1],\Delta_1])^{vec}\right)^{tr}(\overline{\Delta}_1)^{vec}
=\textup{Tr}([[\overline{\Delta}_1,\Delta_1],\Delta_1]\cdot \overline{\Delta}_1^{tr}) \\ \\
=-\textup{Tr}([\Delta_1,[\overline{\Delta}_1,\Delta_1]]\cdot \overline{\Delta}_1^{tr})
=-\textup{Tr}([\Delta_1,\overline{\Delta}_1],\overline{[\Delta_1,\overline{\Delta}_1]}^{tr})
\end{array}
$$
The last computation uses formula \eqref{eqMatVecTens} and
the fact that $R\cdot(\Delta_1)^{vec}=[[\Delta_1,\Delta_1],\overline{\Delta}_1]=0$.

It is well known that the curvature decreases on subbundles,
see, e.g., \cite[lemma (7.14)]{Sch}, thus we obtain the following estimate:
$$
h^{IV}(R^{IV}(\xi,\overline{\xi})\KS(\xi),\KS(\xi))
\leq h^I(R^I(\xi,\overline{\xi})\KS(\xi),\KS(\xi)).
$$
This implies that the holomorphic sectional curvature
$\frac{h^{IV}(R^{IV}(\xi,\overline{\xi})\xi,\xi)}{h^2(\xi,\xi)}$
is smaller than or equal to
$$
-\frac{\textup{Tr}([\Delta_1,\overline{\Delta}_1],\overline{[\Delta_1,\overline{\Delta}_1]}^{tr})}{\textup{Tr}(\Delta_1\cdot\overline{\Delta}^{tr}_1)^2}
$$
As $\xi\in T^{hor}_t(\cblp)$, i.e., $\KS(\partial_{t_1})\in\GIV\subset\GIII$,
the morphism $\KS(\partial_{t_1}):\cL\rightarrow z^{-1}\cL/\cL$ respects the $V$-filtration, and
$[z\KS(\partial_{t_1})]\in{\cE}\!nd_{\cO_\cblp}(\cL/z\cL)$ shifts the (induced) $V$-filtration
by one. By definition, $\Delta_1$ is the matrix representing
$[z\KS(\xi)]\in{\cE}\!nd_\dC((\cL/z\cL)_{|(0,t)})$ with respect to the basis $\underline{v}$,
which shows that it is nilpotent.
Lemma \ref{lemMatrices} then proves that the value of
$-\frac{\textup{Tr}([\Delta_1,\overline{\Delta}_1],\overline{[\Delta_1,\overline{\Delta}_1]}^{tr})}{\textup{Tr}(\Delta_1\cdot\overline{\Delta}^{tr}_1)^2}$
and thus of the holomorphic sectional curvature on $T^{hor}_{\cblp}\backslash \{\textup{zero section}\}$ is
bounded from above by a negative real number.

\end{proof}
\begin{lemma}\label{lemMatrices}
Fix $\mu\in \dN$.
\begin{enumerate}
\item
Consider a matrix $A\in M(\mu\times \mu,\dC)$ which is symmetric and
nilpotent. Then
$$
[A,\overline{A}]=0\iff A=0.
$$
\item
The map
\begin{eqnarray*}
\varphi:\big\{A\in M(\mu\times\mu,\dC)\backslash\{0\} &|& A\textup{ is symmetric and nilpotent}\big\}
\longrightarrow\dR,\\ \\
A &\mapsto& \frac{-\tr\left([A,\overline{A}]\cdot \overline{[A,\overline{A}]}^{tr}\right)}
{\tr\left(A\cdot\overline{A}^{tr}\right)^2}
\end{eqnarray*}
is bounded from above by a negative number.
\end{enumerate}
\end{lemma}
\begin{proof}
\begin{enumerate}
\item
$\Re(A)$ and $\Im(A)$ are real symmetric matrices and thus
dia\-go\-nalizable.
$[A,\overline{A}]=0$ is equivalent to $[\Re(A),\Im(A)]=0$
and to the simultaneous diagonalizability of $\Re(A)$ and $\Im(A)$.
In that case also $A$ is diagonalizable. As $A$ is nilpotent, it vanishes.
\item
The image of the map $\varphi$ does not change when we restrict $\varphi$ to the subset
$$
\{A\in M(\mu\times\mu,\dC)\ |\ A\textup{ is symmetric and nilpotent and }
\tr{(A\cdot \overline{A}}^{tr})=1\}.
$$
This set is compact, its image is contained in $\dR^-$,
$\varphi$ is continuous so that the image is compact and therefore
$\textup{Im}(\varphi)$ has a strictly negative upper bound.
\end{enumerate}
\end{proof}

In the remaining part of this section, we outline some
rather direct consequences of the above curvature calculations.
They are close in spirit to the work of Griffiths and Schmid on the
classifying spaces of Hodge structures (\cite{GSch1, GSch2}). The key tool is the
following result. Let us call a holomorphic map $\phi:M\rightarrow \cblp$
\emph{horizontal} if $d\phi(\cT_M)\subset \phi^*\cT^{hor}_\cblp$, where
$d\phi:\cT_M\rightarrow \phi^*\cT_\cblp$ is the derivative
of $\phi$.
\begin{proposition}
\label{propAhlfors}
\begin{enumerate}
\item
Write $\Delta$ for the open unit disc in $\dC$ and let $\phi:\Delta\rightarrow\cblp$ be a horizontal map. Denote by
$$
\omega_\Delta:=\frac{1}{(1-|r|^2)^2}dr\wedge d\overline{r}
$$
the (metric) $(1,1)$-form associated to the Poincar\'e metric
on $\Delta$ and similarly by $\omega_h$ the form associated
to the metric $h$ on $T_\cblp$ defined above. Then the following inequality holds
$$
c\phi^*\omega_h \leq  \omega_\Delta
$$
for some $c\in\dR_{>0}$ meaning that $\omega_\Delta-c\phi^*\omega_h$ is a positive
semi-definite form.
\item
Let now $M$ be any complex manifold and $\phi:M\rightarrow \cblp$ a horizontal map.
Then $\phi$ is distance-decreasing with respect to the (suitably normalized)
distance $d_h$ on $\cblp$ induced by $h$ and the Kobayashi
pseudo-distance on $M$.
\end{enumerate}
\end{proposition}
\begin{proof}
The proof of the first part is well-known and uses Ahlfors' lemma
(see, e.g., \cite[13.4]{CarlStachPeters}). The second part is an immediate consequence.
\end{proof}
The following rather obvious lemma shows how to apply the above computations
to the study of period mappings.
\begin{lemma}
Let $H$ underly a variation of pure polarized, regular singular TERP-structures on $M$
with constant spectral numbers. Let $\pi:\widetilde{M}\rightarrow M$
be the universal cover. We obtain a period mapping
$$
\phi:\widetilde{M}\longrightarrow \cblp
$$
by associating to $\widetilde{x}\in\widetilde{M}$ the TERP-structure
$H_{|\dC\times\{\pi(x)\}}\in\cblp$. Then we have
$d\phi(\cT_{\widetilde{M}})\subset\phi^*\cT^{hor}_\cblp$, i.e.,
$\phi$ is horizontal.
\end{lemma}
\begin{proof}
The pullback of the universal bundle $(\cL,\nabla)$ under the map $id\times\phi:\dC\times \widetilde{M}\rightarrow \dC\times\cbl$
is isomorphic to $(\cH,\nabla)$. By definition, for a \emph{variation} of TERP-structures, the sheaf $\cH$ is stable
under $z\nabla_X$ for any $X\in (p')^{-1}\cT_{\widetilde{M}}$ (where $p':\dC\times\widetilde{M}
\rightarrow \widetilde{M}$), and not only under $z^n\nabla_X$ as it is the case for $\cL$. Therefore,
$$
(id\times\phi)^*(z\nabla_{d\phi(X)})\cL = (z\nabla_X)(id\times\phi)^*\cL\subset(id\times\phi)^*\cL.
$$
This implies that $\KS(d\phi(\cT_{\widetilde{M}}))$ is contained
in ${\cH}\!om_\cR(\cL,z^{-1}\cL/\cL)$, so that by definition
$\mathit{Im}(d\phi)\subset\phi^*\cT^{hor}_\cblp$.
\end{proof}

As an example of possible applications we give the following rigidity result similar to the
one for variations of Hodge structures.
\begin{corollary}
\label{corRigidTERP}
Let $\TERP$ be a variation of pure
polarized regular singular TERP-structures on $\dC^m$ with constant spectral numbers. Then
the variation is trivial, i.e., the corresponding map
$\phi:\dC^m\rightarrow \cblp$ is constant or, in other words,
$\cH$ is stable under $\nabla$.
\end{corollary}
\begin{proof}
The last lemma and the second point of proposition \ref{propAhlfors} show that the period map
$\phi:\dC^m\rightarrow \cblp$ satisfies $d_{\dC^m}(x,y)\geq
d_h(\phi(x),\phi(y))$, where $d_{\dC^m}$ is the Kobayashi pseudo-distance
on $\dC^m$ and $x,y\in\dC^m$. It is known that $d_{\dC^m}=0$, on the other hand, $d_h$ is a true distance,
so that $\phi$ is necessarily constant.
\end{proof}

Let us finish this paper by pointing out that the above construction
has an a priori unpleasant feature: the metric space $\cblp$ is not complete
in general. We will give a concrete example showing this phenomenon. We will not carry
out all details of the computations which are rather lengthy.

Consider the following topological data: Let $H^\infty_\dR$ be
a three-dimensional real vector space, $H^\infty:=H^\infty_\dR\otimes \dC$ its complexification
and choose a basis $H^\infty=\oplus_{i=1}^3\dC A_i$ such that $\overline{A}_1=A_3$
and $A_2\in H^\infty_\dR$.
Moreover, choose a real number $\alpha_1\in(-3/2,-1)$, put $\alpha_2:=0$, $\alpha_3:=-\alpha_1$
and let $M\in\mathit{Aut}(H^\infty_\dC)$ be given
by $M(\underline{A})=\underline{A}\cdot\diag(\lambda_1,\lambda_2,\lambda_3)$
where $\underline{A}:=(A_1,A_2,A_3)$ and $\lambda_i:=e^{-2\pi i \alpha_i}$.
Putting
$$
\{0\}=F_0^2 \subsetneq
F_0^1:=\dC A_1 \subsetneq
F_0^0 := \dC A_1 \oplus \dC A_2 =
F_0^{-1} \subsetneq
F_0^{-2}:= H^\infty
$$
defines a sum of pure Hodge structures of weights $0$ and $-1$ on $H^\infty_1$ and $H^\infty_{\neq 1}$.
A polarizing form is defined by
$$
S(\underline{A}^{tr},\underline{A}):=
\begin{pmatrix}
0 & 0 & \gamma \\
0 & 1 & 0 \\
-\gamma & 0 &0
\end{pmatrix},
$$
where $\gamma:=\frac{-1}{2\pi i}\Gamma(\alpha_1+2)\Gamma(\alpha_3-1)$. In particular,
we have for $p=1$
$$
i^{p-(-1-p)}S(A_1,A_3)= (-1)iS(A_1,A_3)= \frac{\Gamma(\alpha_1+2)\Gamma(\alpha_3-1)}{2\pi} >0
$$
and for $p=0$
$$
i^{p-(-p)}S(A_2,A_2)=S(A_2,A_2)>0
$$
so that $F_0^\bullet$ indeed induces a pure polarized Hodge structure of weight $-1$ on $H^\infty_{\neq 1}=\dC A_1\oplus\dC A_2$
and a pure polarized Hodge structure of weight $0$ on $H^\infty_1=\dC A_2$.
As $M$ is semi-simple and its eigenspaces are one-dimensional, we have
$\DPMHS=\DcPMHS=\DPHS=\DcPHS=\{F_0^\bullet\}$ and $F^\bullet_0=\widetilde{F}^\bullet_0$.

Let $(H',H_\dR',\nabla)$ be the flat holomorphic bundle on $\dC^*\times\dC$ with real flat subbundle corresponding
to $(H^\infty,H^\infty_\dR,M)$, and put $s_i:=z^{\alpha_i}A_i\in\cH'$.
According to \cite[formula (5.3), (5.4)]{HS1}, the pairing
$P:\cH'\otimes j^*\cH'\rightarrow \cO_{\dC^*\times \dC}$ is determined by the above chosen $S$, namely,
$P(\underline{s}^{tr}, \underline{s}):=(\delta_{i+j,4})_{i,j\in\{1,\ldots,3\}}$.

It follows from the construction in \cite[section 5]{He2} that
the classifying space $D_\mathit{BL}=\cbl$ associated with the given topological
data and the spectrum $\alpha_1,\alpha_2,\alpha_3$ is $\cbl\cong\dC^2=\Spec\dC[r,t]$, with
the universal family of Brieskorn lattices given by $\cH=\oplus_{i=1}^3\cO_{\dC^3}v_i$,
where
$$
\begin{array}{rcl}
v_1 & := & s_1 + r z^{-1} s_2 + \frac{r^2}{2} z^{-2} s_3 + t z^{-1}s_3,\\
v_2 & := & s_2 + r z^{-1} s_3,\\
v_3 & := & s_3.
\end{array}
$$
$\widehat{\cH}$ is pure outside of the real-analytic hypersurface
$(1-\rho)^4-\theta=0$, where $\rho=\frac12r\overline{r}$ and $\theta=t\overline{t}$.

The complement has three components. $\widehat{\cH}$ is polarized on two of them,
those which contain $\{(r,0)\ |\ |r|<\sqrt{2}\}$ and $\{(r,0)\ |\ |r|>\sqrt{2}\}$, respectively.
On the third component the metric on $p_*\widehat \cH$ has signature $(+,-,-)$.
So in this example $\cblp$ has two connected components, one of them is bounded while the other is not.

If we restrict the metric $h$ on $\cT_{\cblp}$ to the tangent space of $\{(r,0)\ |\ |r|\neq \sqrt{2}\}$, then it is
given by
$$
h(\partial_r,\partial_r) = 2\frac{1+\rho^2}{(1-\rho)^4}.
$$

From this it is directly evident that the distance defined by $h$ on the unbounded component of $\cblp$ cannot be complete,
as we have
$$
h(\partial_{r^{-1}},\partial_{r^{-1}})=h(-r^2\partial_r,-r^2\partial_r)
=8\rho^2\frac{1+\rho^2}{(1-\rho)^4}
\stackrel{r\rightarrow \infty}{\longrightarrow} 8\neq\infty.
$$
Comparing the situation to the one for classifying spaces of Hodge
structure (where the distance induced by the Hodge metric on $\DPHS$ is known
to be complete due to the homogeneity of $\DPHS$), it is clear that one needs
to have a complete metric space as a possible target for period maps for
variations of regular singular TERP-structures. We are able to construct such a space, it is
in fact a partial compactification of $\cblp$, on which the metric can be extended and
becomes complete. However, this construction presents a number of technical
difficulties and is somewhat beyond the scope of the present work.
We will treat this and related questions in a subsequent paper.

\bibliographystyle{amsalpha}
\providecommand{\bysame}{\leavevmode\hbox to3em{\hrulefill}\thinspace}
\providecommand{\MR}{\relax\ifhmode\unskip\space\fi MR }
\providecommand{\MRhref}[2]{%
  \href{http://www.ams.org/mathscinet-getitem?mr=#1}{#2}
}
\providecommand{\href}[2]{#2}

\vspace*{1cm}

\nd
Lehrstuhl f\"ur Mathematik VI \\
Institut f\"ur Mathematik\\
Universit\"at Mannheim,
A 5, 6 \\
68131 Mannheim\\
Germany

\vspace*{1cm}

\nd
hertling@math.uni-mannheim.de\\
sevenheck@math.uni-mannheim.de

\end{document}